\documentclass[12pt, reqno]{amsart}
\usepackage{}
\usepackage{stmaryrd}
\usepackage{amsmath}
\usepackage{amsfonts}
\usepackage{amssymb}
\usepackage[all,cmtip]{xy}           %xypic macro for latex2.09
\usepackage{xspace}
\usepackage{bbding}
\usepackage{txfonts}
\usepackage[shortlabels]{enumitem}
\usepackage{cite}

\usepackage{tikz}
\usepackage{titletoc}

\usepackage{ifpdf}
\ifpdf
 \usepackage[colorlinks,final,backref=page,hyperindex]{hyperref}
\else
 \usepackage[colorlinks,final,backref=page,hyperindex,hypertex]{hyperref}
\fi
\usepackage[active]{srcltx}

\makeatletter

\begin{document}
\newcommand {\emptycomment}[1]{} %to remove paragraphs

\baselineskip=15pt
\newcommand{\nc}{\newcommand}
\newcommand{\delete}[1]{}
\nc{\mfootnote}[1]{\footnote{#1}} % Use this to show footnotes
\nc{\todo}[1]{\tred{To do:} #1}

%\delete{
\nc{\mlabel}[1]{\label{#1}}  % Use this to suppress names
\nc{\mcite}[1]{\cite{#1}}  % Use this to suppress names
\nc{\mref}[1]{\ref{#1}}  % Use this to suppress names
\nc{\meqref}[1]{\eqref{#1}} % Use this to suppress names
\nc{\mbibitem}[1]{\bibitem{#1}} % Use this to show number
%}

\delete{
\nc{\mlabel}[1]{\label{#1}  % Use the next two lines to show names
{\hfill \hspace{1cm}{\bf{{\ }\hfill(#1)}}}}
\nc{\mcite}[1]{\cite{#1}{{\bf{{\ }(#1)}}}}  % Use this lines to show names
\nc{\mref}[1]{\ref{#1}{{\bf{{\ }(#1)}}}}  % Use this lines to show names
\nc{\meqref}[1]{\eqref{#1}{{\bf{{\ }(#1)}}}} % Use this lines to show names
\nc{\mbibitem}[1]{\bibitem[\bf #1]{#1}} % Use this to show name
}

\newcommand {\comment}[1]{{\marginpar{*}\scriptsize\textbf{Comments:} #1}}
\nc{\mrm}[1]{{\rm #1}}
\nc{\id}{\mrm{id}}  \nc{\Id}{\mrm{Id}}
%%%%%%%%%%%%%%%%%%%%%%%%

\def\a{\alpha}
\def\b{\beta}
\def\bd{\boxdot}
\def\bbf{\bar{f}}
\def\bF{\bar{F}}
\def\bbF{\bar{\bar{F}}}
\def\bbbf{\bar{\bar{f}}}
\def\bg{\bar{g}}
\def\bG{\bar{G}}
\def\bbG{\bar{\bar{G}}}
\def\bbg{\bar{\bar{g}}}
\def\bT{\bar{T}}
\def\bt{\bar{t}}
\def\bbT{\bar{\bar{T}}}
\def\bbt{\bar{\bar{t}}}
\def\bR{\bar{R}}
\def\br{\bar{r}}
\def\bbR{\bar{\bar{R}}}
\def\bbr{\bar{\bar{r}}}
\def\bu{\bar{u}}
\def\bU{\bar{U}}
\def\bbU{\bar{\bar{U}}}
\def\bbu{\bar{\bar{u}}}
\def\bw{\bar{w}}
\def\bW{\bar{W}}
\def\bbW{\bar{\bar{W}}}
\def\bbw{\bar{\bar{w}}}
\def\btl{\blacktriangleright}
\def\btr{\blacktriangleleft}
\def\c{\cdot}
\def\ci{\circ}
\def\d{\delta}
\def\dd{\diamondsuit}
\def\D{\Delta}
\def\G{\Gamma}
\def\g{\gamma}
\def\l{\lambda}
\def\lr{\longrightarrow}
\def\o{\otimes}
\def\om{\omega}
\def\p{\psi}
\def\r{\rho}
\def\ra{\rightarrow}
\def\rh{\rightharpoonup}
\def\lh{\leftharpoonup}
\def\s{\sigma}
\def\st{\star}
\def\ti{\times}
\def\tl{\triangleright}
\def\tr{\triangleleft}
\def\v{\varepsilon}
\def\vp{\varphi}

%%%%%%%%%%%%%%%%%%%%%%%% Statements
\newtheorem{thm}{Theorem}[section]
\newtheorem{lem}[thm]{Lemma}
\newtheorem{cor}[thm]{Corollary}
\newtheorem{pro}[thm]{Proposition}
\theoremstyle{definition}
\newtheorem{defi}[thm]{Definition}
\newtheorem{ex}[thm]{Example}
\newtheorem{rmk}[thm]{Remark}
\newtheorem{pdef}[thm]{Proposition-Definition}
\newtheorem{condition}[thm]{Condition}
\newtheorem{question}[thm]{Question}
\renewcommand{\labelenumi}{{\rm(\alph{enumi})}}
\renewcommand{\theenumi}{\alph{enumi}}

\nc{\ts}[1]{\textcolor{purple}{Tianshui:#1}}
\nc{\jie}[1]{\textcolor{blue}{LIJIE:#1}}
\font\cyr=wncyr10

%%%%%%%%%%%%%%%%%%%%%%%%%%%%%%%%%%%%%%%%%%%%%%%%%%%%%%%%%%%%%%%%%%
 \title{\bf Double crossed biproducts and related structures}

 \author{Tianshui Ma\textsuperscript{*}}
 \address{School of Mathematics and Information Science, Henan Normal University, Xinxiang 453007, China}
         \email{matianshui@htu.edu.cn}

 \author{Jie Li}
 \address{School of Mathematics and Information Science, Henan Normal University, Xinxiang 453007, China}
         \email{{lijie\_0224@163.com}}

 \author{Haiyan Yang}
 \address{School of Mathematics and Information Science, Henan Normal University, Xinxiang 453007, China}
         \email{yhy3023551288@163.com}

 \author{Shuanhong Wang}
 \address{School of Mathematics, Southeast University, Nanjing 210096, China}
         \email{shuanhwang@seu.edu.cn}

  \thanks{\textsuperscript{*}Corresponding author}

\date{\today}

\begin{abstract}
 Let $H$ be a bialgebra. Let $\s: H\o H\ra A$ be a linear map, where $A$ is a left $H$-comodule coalgebra, and an algebra with a left $H$-weak action $\tl$. Let $\tau: H\o H\ra B$ be a linear map, where $B$ is a right $H$-comodule coalgebra, and an algebra with a right $H$-weak action $\tr$.  In this paper, we improve the necessary conditions for the two-sided crossed product algebra $A\#^{\s} H~{^{\tau}\#} B$ and the two-sided smash coproduct coalgebra $A\times H\times B$ to form a bialgebra (called  double crossed biproduct) such that the condition $b_{[1]}\tl a_0\o b_{[0]}\tr a_{-1}=a\o b$ in Majid's double biproduct (or double-bosonization) is one of the necessary conditions. On the other hand, we provide a more general two-sided crossed product algebra structure via Brzez\'{n}ski's crossed product and give some applications.
\end{abstract}

\keywords{Majid's double biproduct;Brzezi\'{n}ski's crossed product; Radford's biproduct}

\subjclass[2010]{16T05,16T10,16T25}

 \maketitle

\tableofcontents

%\numberwithin{equation}{section}
\allowdisplaybreaks

\section{Introduction and preliminaries}
 Radford's biproduct (\cite{Ra}) is one of the celebrated objects in the theory of Hopf algebras, which also can provide examples for Rota-Baxter bialgebras introduced by Ma and Liu in (\cite{ML}). This plays a fundamental role in the classification of finite-dimensional pointed Hopf algebras (\cite{AS}). Majid realized a categorical interpretation of Radford's biproduct (\cite{Maj1,Maj2}): $A\star H$ is a Radford's biproduct if and only if $A$ is a bialgebra in the Yetter-Drinfel'd category ${}^H_H{\mathbb{YD}}$.

 Let $H$ be a bialgebra, $A$ a bialgebra in ${}^H_H{\mathbb{YD}}$, and $B$ a bialgebra in ${\mathbb{YD}}{}^H_H$.  In \cite{Ma99,Ma97}, Majid gave a sufficient condition
 \begin{eqnarray}
 b_{[1]}\tl a_0\o b_{[0]}\tr a_{-1}=a\o b \label{eq:DB}
 \end{eqnarray}
 for a two-sided smash product algebra $A\# H\# B$ and a two-sided smash coproduct coalgebra $A\times H\times B$ to form a bialgebra, named the Majid's double biproduct (or double-bosonization) and denoted by $A\star H\star B$. Majid's double biproduct provided a powerful tool to construct quantum groups (see \cite{Ma99,Ma97}), some related studies can be found in the literature \cite{HH1,HH2,ML,MLD,MLL,MJS,PVO,SW,Z1,Z4}. As an application of our results in this paper, we find that the condition $(\ref{eq:DB})$ is also necessary in $A\star H\star B$. This is a motivation of writing this paper.

 In \cite{WJZ}, as a generalization of Radford's results in \cite{Ra}, Wang, Jiao and Zhao presented the necessary and sufficient conditions for the (left) crossed product algebra (introduced in \cite{BCM} and extended the smash product algebra) and the (left) smash coproduct coalgebra to be a bialgebra (denoted by $A\star^{\s} H$) and also gave a mapping description. Similarly, in order to extend Majid's double biproduct, Ma and Liu in \cite{MLL} gave the necessary and sufficient conditions for the two-sided crossed product algebra $A\#^{\s} H~{^{\tau}\#} B$ and the two-sided smash coproduct coalgebra $A\times H\times B$ to form a bialgebra (denoted by $A\star^{\s} H~{^{\tau}\star} B$). As can be seen from the notations of these structures, if we substitute $B$ in $A\star H\star B$ (or $A\star^{\s} H~{^{\tau}\star} B$) for the ground field $K$, then $A\star H$ (or $A\star^{\s} H$) is obtained. Comparing the necessary conditions in $A\star H\star B$ and $A\star^{\s} H~{^{\tau}\star} B$, we find that the latter is a little more complicated, especially the condition $(C13)$ in \cite[Theorem 2.1]{MLL} and that the former just contains the necessary and sufficient conditions for the smash product algebra $A\# H$ and the smash coproduct coalgebra $A\times H$ to form a bialgebra $A\star H$. It is worth mentioning that these conditions in the former case are just right for the use of lifting method in the classification of finite-dimensional pointed Hopf algebras. For these reasons, we will improve the necessary conditions in \cite[Theorem 2.1]{MLL} such that the necessary and sufficient conditions for the crossed product algebra $A\#^{\s} H$ and the smash coproduct coalgebra $A\times H$ to form a bialgebra $A\star^{\s} H$ can be included and also $(C13)$ above can be decomposed into several simpler ones (of course containing $(\ref{eq:DB})$). This is another motivation of writing this paper.

 In \cite{Br}, Brzezi\'{n}ski gave the necessary and sufficient conditions for a very extensive crossed product to be an algebra (called Brzezi\'{n}ski's crossed product) which includes the crossed product (\cite{BCM}) and the twisted tensor product (\cite{CIMZ}). We will provide a more general two-sided crossed product algebra structure via Brzez\'{n}ski's crossed product and list some special cases.

 Throughout this paper, we follow the definitions and terminologies in \cite{Ra12}, with all algebraic systems supposed to be over the field $K$. Now, let $C$ be a coalgebra. We use the simple Sweedler's notation for the comultiplication: $\Delta(c)=c_{1}\o c_{2}$ for any $c\in C$. Denote the category of left $H$-comodules by $^H \hbox{Mod}$. For $(M, \rho)\in$ $^H \hbox{Mod}$, write: $\rho(x)=x_{-1}\o x_0 \in H\o M$, for all $x \in M$. Denote the category of right $H$-comodules by $ \hbox{Mod}^H$. For $(M, \p)\in$ $ \hbox{Mod}^H$, write: $\p(x)=x_{[0]}\o x_{[1]} \in M\o H$, for all $x \in M$. We denote the left-left Yetter-Drinfel'd category by ${}^H_H{\mathbb{YD}}$, and the right-right Yetter-Drinfel'd category by ${\mathbb{YD}}{}^H_H$. Given a $K$-space $M$, we write $id_M$ for the identity map on $M$.

 First of all, we rewrite the Brzezi\'{n}ski's crossed product as follows.

 \begin{pro}\label{pro:lbr} ({\it Left Brzezi\'{n}ski's crossed product}) (\cite{Br}): Let $A$ be an algebra, $H$ a vector space and $1_H\in H$. Let $G: H\o H\lr A\o H$ (write $G(x\o x')=x^G\o x'_G$ for all $x, x'\in H$) and $R: H\o A\lr A\o H$ (write $R(x\o a)=a_R\o x_R$ for all $x\in H$ and $a\in A$) be linear maps such that
 $$
 (LB1)~~a_R\o1_{H R}=a\o {1_H},~~~{1_A}_R\o x_R={1_A}\o x
 $$
 holds. Then $A\#^G_R H$ ($=A\o H$ as vector space) is an algebra with unit $1_A\o 1_H$ and a product
 \begin{eqnarray}\label{eq:lbr}
 (a\o x)(a'\o x')=a a'_R{x_R}^G\o x'_G
 \end{eqnarray}
 if and only if the following conditions hold:
 \begin{eqnarray*}
 &(LB2)& (a a')_R\o x_R=a_R a'_r\o x_{R r};\\
 &(LB3)& x^G\o {1_H}_G={1_H}^G\o x_G=1_A\o x;\\
 &(LB4)& {x'^G}_R{x_R}^g\o x''_{G g}=x^G{x'_G}^g\o x''_g;\\
 &(LB5)& a_{R r}{x_r}^G\o x'_{R G}=x^G a_R\o x'_{G R}
 \end{eqnarray*}
 for all $a, a'\in A$, $x, x', x''\in H$, where $G=g$, $R=r$.
 \end{pro}

 \begin{proof} The sufficiency can be gotten by \cite{Br}. So we only prove the necessity.
 By $1_A\o 1_H$ is unit and Eqs.$(LB1)$ and (\ref{eq:lbr}), we can get Eq.$(LB3)$. By the associativity and Eq.(\ref{eq:lbr}), we have
 \begin{eqnarray}\label{eq:lbr-1}
 a{a'_R}{x_R}^G{a''_r}{x'_{Gr}}^g\o x''_g=a(a'a''_R{x'_R}^G)_r{x_r}^g\o x''_{Gg}.
 \end{eqnarray}
 Set $x'=x''=1_H$ in Eq.(\ref{eq:lbr-1}), one can obtain Eq.$(LB2)$. Let $a'=a''=1_A$ in Eq.(\ref{eq:lbr-1}), then Eq.$(LB4)$ holds. Likewise, setting $a'=1_A, x''=1_H$ in Eq.(\ref{eq:lbr-1}), we can get Eq.$(LB5)$. These finish the proof. \end{proof}

 \begin{rmk}\label{rmk:lbr-1}
 When Eqs.$(LB1)$ and $(LB3)$ are interchanged, Proposition \ref{pro:lbr} still holds.
 \end{rmk}

 In what follows, we list some important special cases of left Brzezi\'{n}ski's crossed product.

 \begin{ex}\label{ex:1.2} ({\it Left twisted crossed product}) Let $A$ be an algebra, $H$ a bialgebra, $\s: H\o H\lr A$ and $R: H\o A\lr A\o H$ two linear maps such that Eq.(LB1) holds. Then $A\# ^{\s}_R H$ (=$A\o H$ as a vector space)  is an associative algebra with unit $1_A\o {1_H}$ and multiplication
 $$
 (a\o x)(a'\o x')=a a'_{R} \s(x_{R1}, x'_1)\o x_{R2}x'_2
 $$
 for all $a, a'\in A$ and $x, x'\in H$ if and only if Eq.(LB2) and the following conditions hold:
 \begin{eqnarray*}
 &(LTC3)& \s(x, 1_H)=\s(1_H, x)=\v(x)1_A;\\
 &(LTC4)& \s(x'_1, x''_1)_R\s(x_{R1}, x'_2 x''_2)\o x_{R2}x'_3 x''_3=\s(x_1, x'_1)\s(x_{2}x'_2,  x''_1)\o x_{3}x'_3 x''_2;\\
 &(LTC5)& a_{R r}\s(x_{r1}, x'_{R1})\o x_{r2}x'_{R2}=\s(x_{1}, x'_{1})a_{R}\o (x_{2}x'_{2})_R,
 \end{eqnarray*}
 where $a, a'\in A$, $x, x', x''\in H$ and $r=R$.
 \end{ex}

 \begin{proof} Let $G(x\o x')=\s(x_1, x'_1)\o x_2x'_2$ in $A\#^G_R H$.      \end{proof}

 \begin{ex}\label{ex:1.3} ({\it Left crossed product}) (\cite{BCM}): Let $H$ be a Hopf algebra with a left weak action $\tl$ on the algebra $A$. Let $\s: H\o H\lr A$ be a $K$-linear map. Let $A\#^{\s}_{\tl} H$ be the (usually nonassociative) algebra whose underlying space is $A\o H$ and whose multiplication is given by
 $$
 (a\o x)(a'\o x')=a(x_1\tl a')\s(x_2, x'_1)\o x_3x'_2
 $$
 for all $a, a'\in A$ and $x, x'\in H$. This $A\#^{\s}_{\tl} H$ is an associative algebra with unit $1_A\o {1_H}$ if and only if $\s$ satisfies Eq.(LTC3) and the following conditions:
 \begin{eqnarray*}
 &(LC2)& (x_1\tl \s(x'_1, x''_1))\s(x_2, x'_2x''_2)=\s(x_1,x'_1)\s(x_2x'_2, x'');\\
 &(LC3)& (x_1\tl (x'_1\tl a))\s(x_2, x'_2)=\s(x_1,x'_1)(x_2x'_2\tl a)
 \end{eqnarray*}
 for all $a\in A$ and $x, x', x''\in H$.
 \end{ex}

 \begin{proof} Let $R(x\o a)=x_1 \tl a\o x_2$ in $A\#^{\s}_R H$.      \end{proof}

 \begin{ex}\label{ex:1.4} ({\it Left twisted product}) (\cite{BCM}): Let $H$ be a bialgebra and $A$ an algebra. Let $\s: H\o H\lr A$ be a $K$-linear map. Let $A\#^{\s} H$ (=$A\o H$ as a vector space)  is an associative algebra with unit $1_A\o {1_H}$ and multiplication
 $$
 (a\o x)(a'\o x')=a a'\s(x_1, x'_1)\o x_2 x'_2
 $$
 for all $a, a'\in A$ and $x, x'\in H$ if and only if $\s$ satisfies Eq.(LTC3) and the following conditions:
 \begin{eqnarray*}
 &(L2)& \s(x''_1, x'_1)\s(x, x''_2x'_2)=\s(x_1,x''_1)\s(x_2x''_2, x');\\
 &(L3)&  a \s(x, x')=\s(x, x') a
 \end{eqnarray*}
 for all $a\in A$ and $x, x', x''\in H$.
 \end{ex}

 \begin{proof} Let $x\tl a=\v(x) a$ in $A\#^{\s}_{\tl} H$.      \end{proof}

 \begin{ex}\label{ex:1.5} ({\it Left unified product}) (\cite{AM}):  Let $A$, $H$ be two bialgebras and $\tr: H\o A\lr H$, $\tl: H\o A\lr A$, $\s: H\o H\lr A$ linear maps such that
 $$
 (LU1)~~~~{1_H}\tl a_1\o {1_H}\tr a_2=a\o {1_H};~~x_1\tl 1_A\o x_2\tr 1_A=1_A\o x
 $$
 holds. Then $A\#^{\s}_{\tl \tr} H$ (=$A\o H$ as a vector space)  is an associative algebra with unit $1_A\o {1_H}$ and multiplication
 $$
 (a\o x)(a'\o x')=a(x_1\tl a'_1)\s((x_2\tr a'_2)_1, x'_1) \o (x_2\tr a'_2)_2x'_2
 $$
 for all $a, a'\in A$ and $x, x'\in H$ if and only if the following conditions hold:
 \begin{eqnarray*}
 &(LU2)& x_1\tl (aa')_1\o x_2\tr (aa')_2=(x_1\tl a_1)((x_2\tr a_2)_1\tl a'_1)\o ((x_2\tr a_2)_2\tl a'_2);\\
 &(LU3)& \s(x_1, {1_H})\o x_2=\s({1_H}, x_1)\o x_2=1_A\o x;\\
 &(LU4)& ~~(x_1\tl \s(x'_1, x''_1)_1)\s((x_2\tr \s(x'_1, x''_1)_2)_1, (x'_2x''_2)_1)\o (x_2\tr \s(x'_1, x''_1)_2)_2 (x'_2x''_2)_2\\
 & &=\s(x_1, x'_1)\s((x_2x'_2)_1, x''_1)\o (x_2x'_2)_2 x''_2;\\
 &(LU5)&(x_1\tl (x'_1\tl a_1)_1)\s((x_2\tr (x'_1\tl a_1)_2)_1, (x'_2\tr a_2)_1)\o (x_2\tr (x'_1\tl a_1)_2)_2 \\
 &&\times (x'_2\tr a_2)_2=\s(x_1, x'_1)((x_2x'_2)_1\tl a_1)\o ((x_2x'_2)_2\tr a_2),
 \end{eqnarray*}
 where $a, a'\in A$, $x, x', x''\in H$.
 \end{ex}

 \begin{proof}  Let $R(x\o a)=x_1\tl a_1 \o x_2\tr a_2$ in $A\#^{\s}_R H$.        \end{proof}

 \begin{ex}\label{ex:1.6} ({\it Left twisted tensor product}) (\cite{CIMZ}):  Let $H$, $A$ be two algebras and $R: H\o A\lr A\o H$ a linear map. Then $A\#_R H$ (=$A\o H$ as a vector space)  is an associative algebra with unit $1_A\o {1_H}$ and multiplication
 $$
 (a\o x)(a'\o x')=a{a'_R} \o x_R x'
 $$
 for all $a, a'\in A$ and $x, x'\in H$ if and only if  Eqs.(LB1), (LB2) and the following condition hold:
 \begin{eqnarray*}
 &(LT3)& {a_{R}}\o (x x')_{R}={a}_{R r}\o x_{r}{x'_{R}},
 \end{eqnarray*}
 where $a, a'\in A$, $x, x'\in H$ and $r=R$.
 \end{ex}

 \begin{proof}  Let $G$ be trivial, i.e., $G(x\o x')=1_A\o x x'$ in $A\#^G_R H$.      \end{proof}

 \begin{rmk}\label{rmk:1.7}(1) The well known Drinfeld double is also a special case of left Brzezi\'{n}ski's crossed product.\newline
  \indent{\phantom{\bf Remarks}} (2) From above we know that the left Brzezi\'{n}ski's crossed product is a very general product structure.
 \end{rmk}

 \begin{defi}\label{de:1.9}({\it Majid's double biproduct}) We recall, from \cite{Ma99,Ma97}, the construction of the so-called double biproduct. Let $H$ be a bialgebra, $A$ a bialgebra in ${}^H_H{\mathbb{YD}}$, and $B$ a bialgebra in ${\mathbb{YD}}{}^H_H$. Adopt the following notation for the structure maps: the counits are
 $\v_A$ and $\v_B$, the comultiplications are $\D_A(a)=a_1\o a_2$ and $\D_B(b)=b_1\o b_2$, and the actions and coactions are
 \begin{eqnarray*}
 &&H\o A\lr A,  x\o~~~a\mapsto  x\tl a,~~~A\lr H\o A,~~~a \mapsto  a_{-1}\o a_0,\\
 &&B\o H\lr B,~~~b\o x\mapsto b\tr x,~~~~B\lr  B\o H,~~~b\mapsto  b_{[0]}\o  b_{[1]}
 \end{eqnarray*}
 for all $x\in  H$, $a\in  A$, $b\in B$. The vector space $A\o H\o B$ becomes an algebra (called the two-sided smash product, $A\# H\# B$) with unit $1_A\o 1_H\o 1_B$ and multiplication
 \begin{eqnarray}
 (a\o x\o b)(a' \o x' \o b') = a(x_1\tl a')\o x_2x'_1\o (b\tr  x'_2)b', \label{eq:tssp}
 \end{eqnarray}
 and a coalgebra (called the two-sided smash coproduct, $A\times H\times B$) with
 counit $\v(a\o x\o b)=\v_A(a)\v_H(x)\v_B(b)$ and comultiplication
 \begin{eqnarray}
 &&\D(a\o x\o b)=a_1\o a_{2-1}x_1\o b_{1{[0]}} \o a_{20}\o x_2b_{1{[1]}}\o b_2. \label{eq:tssc}
 \end{eqnarray}
 Moreover, assume that Eq.(\ref{eq:DB}) holds, it follows that $A\o H\o B$ with two-sided smash product algebra and two-sided smash product coalgebra is a bialgebra, called the Majid's double biproduct, denoted by $A\star H\star B$.
 \end{defi}

 \begin{rmk}\label{rmk:1.10} (1) Replacing $A$ (or $B$) with $K$, the double biproduct is exactly the right (or left) variant of Radford's biproduct.

 (2) The double biproduct here is actually the case of Majid's in \cite[Theorem A.1]{Ma99} with a trivial pairing.
 \end{rmk}

\section{Brzezi\'{n}ski's two-sided crossed product}

 In this section, we will first introduce the notion of Brzezi\'{n}ski's two-sided crossed product generalizing the two-sided smash product. First we need to list the right version of some crossed products.

 \subsection{Right Brzezi\'{n}ski's crossed product and its special cases}

 \begin{pro}\label{pro:2.1}({\it Right Brzezi\'{n}ski's crossed product}) (\cite[Theorem 2.1]{Pa14}) Let $H$ be a vector space with a distinguished element $1_H\in H$ and $B$ an algebra. Let $F: H\o H\lr H\o B$ (write $F(x\o x')=x_F\o x'^F$ for all $x, x'\in H$) and $T: B\o H\lr H\o B$ (write $T(b\o x)=x_T\o b_T$ for all $x\in H$ and $b\in B$) be two linear maps such that
 \begin{eqnarray*}
 &(RB1)& 1_{H T}\o b_T=1_H\o b,~x_T\o 1_{B T}=x\o 1_B
 \end{eqnarray*}
 holds. Then $H^F_T\# B$ (=$H\o B$ as a vector space) is an associative algebra with unit $1_H\o {1_B}$ and multiplication
 \begin{eqnarray*}
 (x\o b)(x'\o b')=x_F\o {x'_T}^F b_T b'
 \end{eqnarray*}
 for all $b, b'\in B$ and $x, x'\in H$ if and only if the following conditions hold:
 \begin{eqnarray*}
 &(RB2)& x_{T}\o (b b')_{T}=x_{T t}\o b_{t} b'_{T};\\
 &(RB3)& 1_{H F}\o x^{F}=x_{F}\o {1_H}^{F}=x\o 1_{B};\\
 &(RB4)& x_{F f}\o {x''_{T}}^{f}{x'^{F}}_{T}=x_{f}\o {x'_{F}}^{f}x''^{F};\\
 &(RB5)& x_{T F}\o {x'_{t}}^{F}{b}_{Tt}=x_{F T}\o {b_{T}}x'^{F},
 \end{eqnarray*}
 where $b, b'\in B$, $x, x', x''\in H$ and $f=F, t=T$.
 \end{pro}

 \begin{ex}\label{ex:2.2} (1) {\it (Right twisted crossed product)} Let $B$ be an algebra, $H$ a bialgebra, $\tau: H\o H\lr B$ and $T: B\o H\lr H\o B$ two linear maps such that  Eq.(RB1) holds. Then $H^{\tau}_T\#  B$ (=$H\o B$ as a vector space)  is an associative algebra with unit $1_H\o {1_B}$ and multiplication
 $$
 (x\o b)(x'\o b')=x_1 x'_{T1} \o \tau(x_2, x'_{T2}) b_T b'
 $$
 for all $a, a'\in A$ and $x, x'\in H$ if and only if  Eq.(RB2) and the following conditions hold:
 \begin{eqnarray*}
 &(RTC2)& \tau(x, 1_H)=\tau(1_H, x)=\v(x)1_B;\\
 &(RTC3)& x_1 x'_1 x''_{T1}\o \tau(x_2 x'_2, x''_{T2})\tau(x_3, x'_3)_{T}=x_1 x'_1 x''_{1}\o \tau(x_2,  x'_2x''_{2})\tau(x'_3, x''_3);\\
 &(RTC4)& x_{T1}x'_{t1}\o \tau(x_{T2}, x'_{t2}){b}_{Tt}=(x_1x'_1)_{T}\o {b_{T}}\tau(x_2, x'_2),
 \end{eqnarray*}
 where $b, b'\in B$, $x, x', x''\in H$ and $t=T$.

 \begin{proof} Let $F(x\o x')=x_1 x'_1\o \tau (x_2, x'_2)$ in Proposition \ref{pro:2.1}.      \end{proof}

 (2) {\it (Right unified product)} Let $B$, $H$ be two bialgebras and $\btl: B\o H\lr H$, $\btr: B\o H\lr B$, $\tau: H\o H\lr B$ linear maps such that
 $$
 (RU1)~~~~b_1\btl {1_H}\o b_2\btr {1_H}={1_H}\o b;~~1_B\btl x_1\o 1_B\btr x_2=x\o 1_B
 $$
 holds. Then $H^{\tau}_{\btr \btl}\# B$ (=$H\o B$ as a vector space)  is an associative algebra with unit ${1_H}\o 1_B$ and multiplication
 $$
 (x\o b)(x'\o b')=x_1(b_1\btl x'_1)_1 \o \tau(x_2, (b_1\btl x'_1)_2)(b_2\btr x'_2)b'
 $$
 for all $x, x'\in H$ and $b, b'\in B$ if and only if the following conditions hold:
 \begin{eqnarray*}
 &(RU2)& b_1b'_1\btl x_1\o b_2b'_2\btr x_2=b_1\btl (b'_1\btl x_1)_1\o (b_2\btr (b'_1\btl x_1)_2)(b'_2\btr x_2);\\
 &(RU3)& x_1\o \tau(1_H, x_2)=x_1\o \tau(x_2, 1_H)=x\o 1_B;\\
 &(RU4)& ~~x_1x'_1(\tau(x_3, x'_3)_1\btl x''_1)_1\o \tau(x_2x'_2, (\tau(x_3, x'_3)_1\btl x''_1)_2)(\tau(x_3, x'_3)_2\btl x''_2)\\
 & &=x_1x'_1x''_1\o \tau(x_2, x'_2x''_2)\tau(x'_3, x''_3);\\
 &(RU5)&(b_1\btl x_1)_1((b_2\btr x_2)_1\btl x'_1)_1\o \tau((b_1\btl x_1)_2, ((b_2\btr x_2)_1\btl x'_1)_2) \\
 &&\times ((b_2\btr x_2)_2\btr x'_2)=b_1\btl x_1x'_1\o (b_2\btr x_2x'_2)\tau(x_3, x'_3),
 \end{eqnarray*}
 where $b, b'\in B$, $x, x', x''\in H$.

 \begin{proof}  Let $T(b\o x)=b_1\btl x_1 \o b_2\btr x_2$ in $H^{\tau}_T\#  B$.      \end{proof}

 (3) {\it (Right $F$-twist product)} Let $H$ be a vector space with a distinguished element $1_H\in H$ and $B$ an algebra. Let $F: H\o H\lr A\o H$ (write $F(x\o x')=x_F\o x'^F$ for all $x, x'\in H$) be a linear maps. Then $H^F\# B$ (=$H\o B$ as a vector space)  is an associative algebra with unit $1_H\o {1_B}$ and multiplication
 $$
 (x\o b)(x'\o b')=x_F\o {x'}^F b b'
 $$
 for all $b, b'\in B$ and $x, x'\in H$ if and only if  Eq.(RB3) and the following conditions hold:
 \begin{eqnarray*}
 &(RF2)& x_{F f}\o {x''}^{f}{x'^{F}}=x_{f}\o {x'_{F}}^{f}x''^{F};\\
 &(RF3)& x_{F}\o {x'}^{F}{b}=x_{F}\o {b}x'^{F},
 \end{eqnarray*}
 where $b, b'\in B$, $x, x', x''\in H$ and $f=F$.

 \begin{proof} Let $T$ be trivial in Proposition \ref{pro:2.1}.    \end{proof}

 Here we remark that if we set $F(x, x')=x_1 x'_1\o \tau (x_2, x'_2)$ in $H^F\# B$, then we get the {\it right twist product} $H^{\tau}\# B$ and the conditions $(RB3), (RF2)$, $(RF3)$ become the conditions $(RTC2)$ and
 \begin{eqnarray*}
 &(R2)& \tau(x_1x'_1, x'')\tau(x_2, x'_2)=\tau(x, x'_1x''_1)\tau(x'_2, x''_2);\\
 &(R3)& \tau(x, x')b=b\tau(x, x'),
 \end{eqnarray*}
 which is the right version of $(LTC3), (L2)$, $(L3)$.

 (4) {\it(Right twisted tensor product)} Let $H$, $B$ be two algebras and $T: B\o H\lr H\o B$ (write $T(b\o x)=x_T\o b_T$ for all $x\in H$ and $b\in B$) a linear map. Then $H_T\# B$ (=$H\o B$ as a vector space)  is an associative algebra with unit $1_H\o {1_B}$ and multiplication
 $$
 (x\o b)(x'\o b')=x{x'_T} \o b_T b'
 $$
 for all $b, b'\in B$ and $x, x'\in H$ if and only if  Eqs.(RB1), (RB2) and the following condition hold:
 \begin{eqnarray*}
 &(RT3)& x_{T}{x'_{t}}\o {b}_{Tt}=(x x')_{T}\o {b_{T}},
 \end{eqnarray*}
 where $b, b'\in B$, $x, x'\in H$ and $t=T$.

 \begin{proof}  Let $F$ be trivial, i.e., $F(x\o x')=x x'\o 1_B$ in Proposition \ref{pro:2.1}.             \end{proof}

 (5) {\it (Right crossed product)} (\cite{MLL}) Let $H$ be a Hopf algebra with a right weak action $\tr$ on the algebra $B$. Let $\tau: H\o H\lr B$ be a $K$-linear map. Let $H^{\tau}_{\tr}\# B$ be the (usually nonassociative) algebra whose underlying space is $H\o B$ and whose multiplication is given by
 $$
(x\o b)(x'\o b')=x_1x'_1\o \tau(x_2, x'_2)(a\tr x'_3)b'
 $$
 for all $b, b'\in B$ and $x, x'\in H$. This $H^{\tau}_{\tr}\# B$ is an associative algebra with unit $1_H\o {1_B}$ if and only if $\tau$ satisfies Eq.(RTC2) and the following conditions:
 \begin{eqnarray*}
 &(RC2)&\tau(x_1x'_1, x''_1)(\tau(x_2, x'_2)\tr x''_2)=\tau(x, x'_1x''_1)\tau(x'_2, x''_2);\\
 &(RC3)&\tau(x'_1, x''_1)((b\tr x'_2)\tr x''_2)=(b\tr x'_1x''_1)\tau(x'_2, x''_2)
 \end{eqnarray*}
 for all $b\in B$ and $x, x', x''\in H$.

 \begin{proof}  Let $T(b\o x)=x_1\o b\tr x_2$ in $H^{\tau}_T\#  B$.                                    \end{proof}

 \end{ex}

 \subsection{Brzezi\'{n}ski's two-sided crossed product algebra and its special cases}
 Next we give the main result in this section.

 \begin{thm}\label{thm:2.3} Let $H$ be a coalgebra with an element $1_H$ such that $\D(1_H)=1_H\o 1_H$ and $A$, $B$ two algebras. Let $G: H\otimes H \longrightarrow A\otimes H$, $R: H\otimes A \longrightarrow A\otimes H$, $T: B\o H \lr H\o B$ and $\tau: H\otimes H \longrightarrow B$ be linear maps, such that Eqs. $(LB1)$-$(LB3)$, $(RB1)$, $(RB2)$, $(RTC2)$ hold. Then $A\#{_{R}^{G} H^{\tau}_{T}}\# B$ $(=A\o H\o B$ as vector space) is an associative algebra with unit $1_A\o 1_H\o {1_B}$ and multiplication
 $$
 (a\o x\o b)(a'\o x' \o b')=a a'_{R}{x_{R1}}^{G}\otimes x'_{T1G}\otimes \tau(x_{R2},x'_{T2})b_{T}b'
 $$
 for all $a, a'\in A$, $x, x'\in H$ and $b, b'\in B$ if and only if the following conditions hold:
 \begin{eqnarray*}
 &(BT1)&~~{x_{1}}^{G}{x'_{1G1}}^{g}\otimes x''_{T1g}\otimes \tau(x'_{1G2},x''_{T2}){\tau(x_{2},x'_{2})}_{T}\\
 &&~~~~~~~={{x'_{1}}^{G}}_{R}{x_{R1}}^{g}\otimes x''_{1G1g}\otimes \tau(x_{R2},x''_{1G2})\tau(x'_{2},x''_{2});\\
 &(BT2)&~~a_{Rr}{x_{r1}}^{G}\otimes x'_{R1G}\otimes \tau(x_{r2},x'_{R2})\\
 &&~~~~~~~={x_{1}}^{G}a_{R}\otimes x'_{1GR}\otimes \tau(x_{2},x'_{2});\\
 &(BT3)&~~{x_{1}}^{G}\otimes x'_{1GT}\otimes b_{T}\tau(x_{2},x'_{2})\\
 &&~~~~~~~={x_{T1}}^{G}\otimes x'_{t1G}\otimes \tau(x_{T2},x'_{t2})b_{Tt};\\
 &(BT4)&~~a_{R}\otimes x_{RT}\otimes b_{T}=a_{R}\otimes x_{TR}\otimes b_{T}.
 \end{eqnarray*}
 where $a\in A$, $x, x', x''\in H$, $b\in B$ and $R=r, G=g, T=t$. In this case, we call the algebra $A\#{_{R}^{G} H^{\tau}_{T}}\# B$ {\it Brzezi\'{n}ski's two-sided crossed product algebra}.
 \end{thm}

 \begin{proof} $(\Longleftarrow)$ We check the associativity as follows: for all $a, a', a''\in A$, $x, x', x''\in H$ and $b, b', b''\in B$,
 \begin{eqnarray*}
 &&((a\o x\o b)(a'\o x' \o b'))(a''\o x'' \o b'')\\
 &=&aa'_{R}{x_{R 1}}^{g}a''_{r}{x'_{T 1 g r 1}}^{G}\otimes x''_{t 1 G}\otimes \tau(x'_{T 1 g r 2},x''_{t 2})(\tau(x_{R 2},x'_{T 2})b_{T}b')_{t}b''\\
 &\stackrel{(RB2)}{=} &aa'_{R}{x_{R 1}}^{g}a''_{r}{x'_{T 1 g r 1}}^{G}\otimes x''_{t \bT \bt 1G}\otimes \tau(x'_{T 1 g r 2},x''_{t \bT \bt 2}){\tau(x_{R 2},x'_{T 2})}_{\bt}b_{T \bT}b'_{t}b''\\
 &\stackrel{(BT2)}{=} &aa'_{R}a''_{r \bR}{x_{R \bR 1}}^{g}{x'_{T r 1 g 1}}^{G}\otimes x''_{t \bT \bt 1 G}\otimes \tau(x'_{T r 1 g 2},x''_{t \bT \bt 2}){\tau(x_{R \bR 2},x'_{T r 2})}_{\bt}b_{T \bT}b'_{t}b''\\
 &\stackrel{(BT4)}{=} &aa'_{R}a''_{r \bR}{x_{R \bR 1}}^{g}{x'_{r T 1 g 1}}^{G}\otimes x''_{t \bT \bt 1 G}\otimes \tau(x'_{r T 1 g 2},x''_{t \bT \bt 2})\tau(x_{R \bR 2},x'_{r T 2})_{\bt}b_{T \bT}b'_{t}b''
 \end{eqnarray*}
 and
 \begin{eqnarray*}
 &&(a\o x\o b)((a'\o x' \o b')(a''\o x'' \o b''))\\
 &=&a(a'a''_{r}{x'_{r 1}}^{g})_{R}{x_{R 1}}^{G}\otimes x''_{t 1 g T 1 G}\otimes \tau(x_{R 2},x''_{t 1 g T 2})b_{T}\tau(x'_{r 2},x''_{t 2})b'_{t}b''\\
 &\stackrel{(LB2)}{=} &aa'_{R}(a''_{r}{x'_{r 1}}^{g})_{\bR}{x_{R \bR 1}}^{G}\otimes x''_{t 1 g T 1 G}\otimes \tau(x_{R \bR 2},x''_{t 1 g T 2})b_{T}\tau(x'_{r 2},x''_{t 2})b'_{t}b''\\
 &\stackrel{(BT3)}{=} &aa'_{R}(a''_{r}{x'_{r \bT 1}}^{g})_{\bR}{x_{R \bR1}}^{G}\otimes x''_{t T 1 g 1 G}\otimes \tau(x_{R \bR 2},x''_{t T 1 g 2})\tau(x'_{r \bT 2},x''_{t T 2})b_{\bT T}b'_{t}b''\\
 &\stackrel{(LB2)}{=} &aa'_{R}a''_{r \bR}{{x'_{r \bT 1}}^{g}}_{\br}{x_{R \bR \br 1}}^{G}\otimes x''_{t T 1 g 1 G}\otimes \tau(x_{R \bR \br 2},x''_{t T 1 g 2})\tau(x'_{r \bT 2},x''_{t T 2})b_{\bT T}b'_{t}b''\\
 &\stackrel{(BT1)}{=} &aa'_{R}a''_{r \bR}{x_{R \bR 1}}^{g}{x'_{r \bT 1 g 1}}^{G}\otimes x''_{t T \bt 1 G}\otimes \tau(x'_{r \bT 1 g 2},x''_{t T \bt 2}){\tau(x_{R \bR 2},x'_{r \bT 2})}_{\bt}b_{\bT T}b'_{t}b'',
 \end{eqnarray*}
 so
 $$
 ((a\o x\o b)(a'\o x' \o b'))(a''\o x'' \o b'')=(a\o x\o b)((a'\o x' \o b')(a''\o x'' \o b'')).
 $$
 We check the unit as follows: for all $a\in A$, $x\in H$ and $b\in B$,
 \begin{eqnarray*}
 (a\o x\o b)(1_A\o 1_H\o 1_B)
 &=&a1_{A R}{x_{R 1}}^{G}\o 1_{H T 1 G}\o \tau(x_{R 2},1_{H T 2})b_T\\
 &\stackrel{(LB1)}{=} &a{x_{1}}^{G}\o 1_{H T 1 G}\o \tau(x_{2},1_{H T 2})b_T\\
 &\stackrel{(RB1)}{=} &a{x_{1}}^{G}\o 1_{H G}\o \tau(x_{2},1_{H})b\\
 &\stackrel{(LB3)}{=} &a\o x_1\o \tau(x_{2},1_{H})b\\
 &\stackrel{(RTC2)}{=} &a\o x\o b.
 \end{eqnarray*}
 \begin{eqnarray*}
 (1_A\o 1_H\o 1_B)(a\o x\o b)
 &=&a_R{1_{H R 1}}^{G}\o x_{T 1 G}\o \tau(1_{H R 2},x_{T 2})1_{B T}b\\
 &\stackrel{(LB1)}{=} &a{1_{H}}^{G}\o x_{T 1 G}\o \tau(1_{H},x_{T 2})1_{B T}b\\
 &\stackrel{(RB1)}{=} &a{1_{H}}^{G}\o x_{1 G}\o \tau(1_{H},x_{2})b\\
 &\stackrel{(LB3)}{=} &a\o x_{1}\o \tau(1_{H},x_{2})b\\
 &\stackrel{(RTC2)}{=} &a\o x\o b.
 \end{eqnarray*}

 $(\Longrightarrow)$ It is straightforward. \end{proof}

 \begin{rmk}\label{rmk:2.4} Taking either $A=K$ or $B=K$, we obtain the right Brzezi\'{n}ski crossed product and left twisted crossed product, respectively.
 \end{rmk}

 \begin{ex}\label{ex:2.5} (1) {\it (Two-sided twisted crossed product)} Let $H$ be a bialgebra and $A$, $B$ two algebras. Let $\s: H\o H\lr A$, $\tau: H\o H\lr B$, $R: H\o A\lr A\o H$ and $T: B\o H \lr H\o B$ be linear maps such that Eqs.$(LB1)$, $(LB2)$, $(RB1)$, $(RB2)$, $(LTC3)$, $(RTC3)$ hold. Then $A\#{_{R}^{\ \s} H^{\tau}_{T}}\# B$  $(=A\o H\o B$ as vector space)  is an associative algebra with unit $1_A\o 1_H\o {1_B}$ and multiplication
 $$
 (a\o x\o b)(a'\o x' \o b')=a a'_{R}\s(x_{R1}, x'_{T1})\o x_{R2}x'_{T2}\o \tau(x_{R3}, x'_{T3})b_{T}b'
 $$
 for all $a, a'\in A$, $x, x'\in H$ and $b, b'\in B$ if and only if the conditions  $(BT4)$ and
 \begin{eqnarray*}
 &(TC1)&\s(x_1, x'_1)a_{R}\o (x_2x'_2)_{R}\o \tau(x_3, x'_3)=a_{R r}\s(x_{r 1}, x'_{R 1})\o x_{r 2}x'_{R 2}\o \tau(x_{r 3}, x'_{R 3})\\
 &(TC2)&\s(x_1, x'_1)\s(x_2x'_2, x''_{T 1})\o x_3x'_3x''_{T 2}\o \tau(x_4x'_4, x''_{T 3})\tau(x_5, x'_5)_{T}\\
 &&~~~~~~~~=\s(x'_1, x''_1)_{R}\s(x_{R 1}, x'_2x''_2)\o x_{R 2}x'_3x''_3\o \tau(x_{R 3}, x'_4x''_4)\tau(x'_5, x''_5)\\
 &(TC3)&\s(x_{T 1}, x'_{t 1})\o x_{T 2}x'_{t 2}\o \tau(x_{T 3}, x'_{t 3})b_{T t}=\s(x_{1}, x'_{1})\o (x_{2}x'_{2})_{T}\o b_{T}\tau(x_{3}, x'_{3})
 \end{eqnarray*}
 are satisfied for all $a\in A$, $x, x', x''\in H$, $b, b'\in B$ and $R=r, T=t$.

 \begin{proof}  Let $G(x\o x')=\sigma (x_1, x'_1)\o x_2 x'_2$ in Theorem \ref{thm:2.3}.        \end{proof}

 (2) ({\it Two-sided twisted product}) Let $H$ be a bialgebra and $A$, $B$ two algebras. Let $\s: H\o H\lr A$, $\tau: H\o H\lr B$ be linear maps and $A\#^{\s} H$ left twisted product, $H^{\tau}\# B$ right twisted product.  Then $A\#{^{\s} H^{\tau}}\# B$  $(=A\o H\o B$ as vector space)  is an associative algebra with unit $1_A\o 1_H\o {1_B}$ and multiplication
 $$
 (a\o x\o b)(a'\o x' \o b')=a a'\s(x_{1}, x'_{1})\o x_{2}x'_{2}\o \tau(x_{3}, x'_{3})b b'
 $$
 for all $a, a'\in A$, $x, x'\in H$ and $b, b'\in B$.

 \begin{proof}  Let $R$ and $T$ be trivial in $A\#{_{R}^{\ \s} H^{\tau}_{T}}\# B$. And the rest is straightforward.        \end{proof}

 (3) ({\it Two-sided twisted tensor product}) Let $H$, $A$, $B$ be three algebras. Let $R: H\o A\lr A\o H$, $T: B\o H \lr H\o B$ be linear maps and $A\#_{R} H$ left twisted tensor product, $H_{T}\# B$ right twisted tensor product. Then $A\#{_{R} H_{T}}\# B$  $(=A\o H\o B$ as vector space)  is an associative algebra with unit $1_A\o 1_H\o {1_B}$ and multiplication
 $$
 (a\o x\o b)(a'\o x' \o b')=a a'_{R}\o x_{R}x'_{T}\o b_{T}b'
 $$
 for all $a, a'\in A$, $x, x'\in H$ and $b, b'\in B$ if and only if Eq.(BT4) holds.

 \begin{proof}  We only prove the necessity of $(BT4)$ as follows.  And the rest is straightforward by letting $\s$ and $\tau$ be trivial in $A\#{_{R}^{\ \s} H^{\tau}_{T}}\# B$.
 By  the associativity in $A\#{_{R} H_{T}}\# B$, we can get
 \begin{eqnarray*}
 a a'_{R}a''_{r}\o (x_{R} x'_{T})_ r x''_{t}\o (b_{T}b')_{t}b''=a (a'a''_{R})_{r}\o x_{r}(x'_{R}x''_{T})_t\o b_{t}b'_{T}b''.
 \end{eqnarray*}
 Let $a=a'=1_A$, $b'=b''=1_B$ and $x'=x''=1_H$ in the above equation. Then we obtain $(BT4)$.       \end{proof}

 (4) ({\it Two-sided crossed product}) (\cite{MLL}) Let $H$ be a bialgebra with a left weak action $\tl$ on the algebra $A$ and a right weak action $\tr$ on the algebra $B$. Let $\s: H\o H\ra A$ and $\tau: H\o H\ra B$ be linear maps and  $A\#_{\tl}^{\s} H$ left crossed product, $H^{\tau}_{\tr}\# B$ right crossed product. Then $A\#{_{\tl}^{\s} H^{\tau}_{\tr}}\# B$  $(=A\o H\o B$ as vector space)  is an associative algebra with unit $1_A\o 1_H\o {1_B}$ and multiplication
 $$
 (a\o x\o b)(a'\o x' \o b')=a(x_1\tl a')\s(x_2, x'_1)\o x_3x'_2\o \tau(x_4,x'_3)(b\tr x'_4)b'
 $$
 for all $a, a'\in A$, $x, x'\in H$ and $b, b'\in B$.
 \end{ex}

 \begin{proof} The result can be proved by letting $R(a\o x)=x_1\tl a\o x_2$ and $T(b\o x)=x_1\o b\tr x_2$ in Example \ref{ex:2.5} (1).   \end{proof}

 \begin{rmk}\label{rmk:2.6} (1) When $\s$ and $\tau$ are trivial in $A\#{_{\tl}^{\s} H^{\tau}_{\tr}}\# B$, we can obtain the two-sided smash product $A\# H\# B$ in \cite{Ma99,Ma97}.

 (2) Two-sided twisted tensor product $A\#{_{R} H_{T}}\# B$ in Example \ref{ex:2.5} (3) is exactly the iterated twisted tensor product in \cite[Theorem 2.1]{MLPVO} with the map $B\o A\lr A\o B$ trivial. In this case, Eq.(BT4) is exactly \cite[Eq.(2.1)]{MLPVO}. We note that Eq.(2.1) in \cite[Theorem 2.1]{MLPVO} is only a sufficient condition for $A\#{_{R} H_{T}}\# B$ to be an associative algebra. Here Eq.(BT4) is also necessary.

 (3) Two-sided twisted tensor product $A\#{_{R} H_{T}}\# B$ in Example \ref{ex:2.5} (3) can be seen as the iterated crossed product in \cite[Theorem 2.3]{Pa14} with the structure maps $Q: B\o A\lr A\o H\o B, b\o a\mapsto a\o 1_H\o b$, $\upsilon: A\o A\lr A\o H, a\o a'\mapsto a a'\o 1_H$ and $\sigma: B\o B\lr H\o B, b\o b'\mapsto 1_H\o b b'$. In this case, Eq.(BT4) is exactly \cite[Eq.(2.9)]{Pa14} and \cite[Eqs.(2.10) and (2.11)]{Pa14} hold automatically.  Eq.(2.9) in \cite[Theorem 2.3]{Pa14} is just one sufficient condition for $A\#{_{R} H_{T}}\# B$ to be an associative algebra.
 \end{rmk}

 \subsection{Brzezi\'{n}ski's two-sided crossed product bialgebra}

  \begin{thm} \label{thm:3.1} Let $A$, $B$ be bialgebras and $H$ a coalgebra with an element $1_H$ such that $\D(1_H)=1_H\o 1_H$. Let $G: H\otimes H \longrightarrow A\otimes H$, $R: H\otimes A \longrightarrow A\otimes H$, $T: B\o H \lr H\o B$ and $\tau: H\otimes H \longrightarrow B$ be linear maps such that $\v_B(\tau(x,x'))=\v_H(x)\v_H(x')$. Then $A\#{_{R}^{G} H^{\tau}_{T}}\# B$ equipped with the two-sided tensor product coalgebra $A\o H\o B$ becomes a bialgebra if and only if the following conditions hold:
  \begin{eqnarray*}
  &(G1)& \v_{A}(x^{G})\v_{H}({x'}_{G})=\v_{H}(x)\v_{H}(x');\\
  &(G2)&\v_{A}(a_{R})\v_{H}(x_{R})=\v_{A}(a)\v_{H}(x);
  \v_{B}(b_{T})\v_{H}(x_{T})=\v_{B}(b)\v_{H}(x);\\
  &(G3)& {a}_{R1}\otimes x_{R1}\otimes {a}_{R2}\otimes x_{R2}={a}_{1R}\otimes x_{1R}\otimes {a}_{2r}\otimes x_{2r};\\
  &(G4)& x_{T 1}\otimes b_{T 1}\otimes x_{T 2}\otimes b_{T 2}=x_{1 T}\otimes b_{1 T}\otimes x_{2 t}\otimes b_{2 t};\\
  &(G5)& {{x_{1}}^{G}}_{1}\otimes x'_{1 G 1}\otimes {\tau(x_{2},x'_{2})}_{1}\otimes {{x_{1}}^{G}}_{2}\otimes x'_{1 G 2}\otimes {\tau(x_{2},x'_{2})}_{2}\\
  &&~~~~~~~={x_{1}}^{G}\otimes x'_{1 G}\otimes \tau(x_{2},x'_{2})\otimes {x_{3}}^{g}\otimes x'_{3 g}\otimes \tau(x_{4},x'_{4}),
  \end{eqnarray*}
  where $a \in A$, $x, x'\in H$ and $b \in B$. In this case, we call this bialgebra {\it Brzezi\'{n}ski's two-sided crossed product bialgebra} and denoted by $A\star {_{R}^{G} H^{\tau}_{T}}\star B$.
  \end{thm}

  \begin{proof} $(\Longleftarrow)$ First by Eqs.(G1) and (G2), we can show that $\v_{A\o H\o B}$ is an algebra map. Next, we check that $\D_{A\o H\o B}$ is an algebra map. In fact, for all $a, a'\in A$, $x, x'\in H$ and $b, b'\in B$, we have
  \begin{eqnarray*}
 &&\D_{A\o H\o B}((a\o x\o b)(a'\o x'\o b'))\\
 &=&\D_{A\o H\o B}(aa'_{R}{x_{R 1}}^{G}\otimes x'_{T 1 G}\otimes \tau(x_{R2},x'_{T2})b_{T}b')\\
 &=& a_{1}a'_{R 1}{{x_{R 1}}^{G}}_{1}\otimes x'_{T 1 G 1}\otimes {\tau(x_{R 2},x'_{T 2})}_{1}b_{T 1}b'_{1}\otimes a_{2}a'_{R 2}{{x_{R 1}}^{G}}_{2}\otimes x'_{T 1 G 2}\otimes {\tau(x_{R 2},x'_{T 2})}_{2}b_{T 2}b'_{2}\\
 &\stackrel{(G5)}{=}& a_{1}a'_{R 1}{x_{R 1 1}}^{G}\otimes x'_{T 1 1 G }\otimes \tau(x_{R 1 2},x'_{T 1 2})b_{T 1}b'_{1}\otimes a_{2}a'_{R 2}{x_{R 2 1}}^{g}\otimes x'_{T 2 1 g}\otimes \tau(x_{R 2 2},x'_{T 2 2})b_{T 2}b'_{2}\\
 &\stackrel{(G3)}{=}& a_{1}a'_{1 R}{x_{1 R 1}}^{G}\otimes x'_{T 1 1 G }\otimes \tau(x_{1 R 2},x'_{T 1 2})b_{T 1}b'_{1}\otimes a_{2}a'_{2 r}{x_{2 r 1}}^{g}\otimes x'_{T 2 1 g}\otimes \tau(x_{2 r 2},x'_{T 2 2})b_{T 2}b'_{2}\\
 &\stackrel{(G4)}{=}& a_{1}a'_{1 R}{x_{1 R 1}}^{G}\otimes x'_{1 T 1 G }\otimes \tau(x_{1 R 2},x'_{1 T 2})b_{1 T}b'_{1}\otimes a_{2}a'_{2 r}{x_{2 r 1}}^{g}\otimes x'_{2 t 1 g}\otimes \tau(x_{2 r 2},x'_{2 t 2})b_{2 t}b'_{2}\\
 &=&\D_{A\o H\o B}(a\o x\o b)\D_{A\o H\o B}(a'\o x'\o b')
 \end{eqnarray*}
 and
 \begin{eqnarray*}
  \D_{A\o H\o B}(1_A\o 1_H\o 1_B)
  &=&1_{A 1}\o 1_{H 1}\o 1_{B 1}\o 1_{A 2}\o 1_{H 2}\o 1_{B 2}\\
  &=&1_{A}\o 1_{H}\o 1_{B}\o 1_{A}\o 1_{H}\o 1_{B}.
 \end{eqnarray*}

 $(\Longrightarrow)$ Since $\v_{A\times H\times B}$ is an algebra map, then we have
 \begin{eqnarray*}
 &(BA2)&~~ \varepsilon_{A}(aa'_{R}{x_{R 1}}^{G})\varepsilon_{H}(x'_{T 1 G})\varepsilon_{B}(\tau(x_{R2},x'_{T2})b_{T}b')=\varepsilon_{A}(a)
 \varepsilon_{H}(x)\varepsilon_{B}(b)\varepsilon_{A}(a')\varepsilon_{H}(x')
 \varepsilon_{B}(b').
\end{eqnarray*}

 Let $a=a'=1_{A}$,$b=b'=1_{B}$ in Eq.$(BA2)$, we obtain $(G1)$. Similarly, $(G2)$ holds.

 By $\D((a\o x\o b)(a'\o x'\o b'))=\D(a\o x\o b)\D(a'\o x'\o b')$, we have
 \begin{eqnarray*}
 &(BA3)&~~ (aa'_{R}{x_{R 1}}^{G})_{1}\otimes x'_{T 1 G 1}\otimes (\tau(x_{R2},x'_{T2})b_{T}b')_{1}\otimes (aa'_{R}{x_{R 1}}^{G})_{2}\otimes x'_{T 1 G 2}\otimes (\tau(x_{R2},x'_{T2})b_{T}b')_{2}\\
 &&~~~~~~~=a_{1}a'_{1 R}{x_{1 R 1}}^{G}\otimes x'_{1 T 1 G }\otimes \tau(x_{1 R 2},x'_{1 T 2})b_{1 T}b'_{1}\otimes a_{2}a'_{2 r}{x_{2 r 1}}^{g}\otimes x'_{2 t 1 g}\otimes \tau(x_{2 r 2},x'_{2 t 2})b_{2 t}b'_{2}.
 \end{eqnarray*}
 Let $a=1_{A},b=b'=1_{B},x'=1_{H}$ in Eq.$(BA3)$, we have
 \begin{eqnarray*}
 &&(a'_{R}{x_{R 1}}^{G})_{1}\otimes 1_{H T 1 G 1}\otimes (\tau(x_{R 2},1_{H T 2})1_{B T})_{1}\otimes (a'_{R}{x_{R 1}}^{G})_{2}\otimes 1_{H T 1 G 2}\\
 &&\otimes (\tau(x_{R 2},1_{H T 2})1_{B T})_{2}=a'_{1 R}{x_{1 R 1}}^{G}\otimes 1_{H T 1 G }\otimes \tau(x_{1 R 2},1_{H T 2})1_{B T}\\
 &&\otimes a'_{2 r}{x_{2 r 1}}^{g}\otimes 1_{H t 1 g}\otimes \tau(x_{2 r 2},1_{H t 2})1_{B t}.
 \end{eqnarray*}
 Then we can obtain $(G3)$ by $(LB1), (LB3)$ and $(RB1), (RTC2)$. Liewise, one has $(G4)$ and $(G5)$.\end{proof}

 \begin{defi}\label{de:3.2} Let $A, B$ be two algebras, and $H$ be a coalgebra with a distinguished element $1_H\in H$. Let $G:H\o H\longrightarrow A\o H$ , $\tau:H\o H\longrightarrow B$ and $ S:H\longrightarrow H$ be linear maps. Then $S$ is called a {\it $(G,\tau)$-antipode of $H$} if for all $x\in H$, the following conditions hold:
 \begin{eqnarray*}
 (I1)~~{S(x_1)_1}^{G}\o x_{2 G}\o \tau(S(x_1)_2,x_3)=1_A\o 1_H\o 1_B\v_H(x);\\
 (I2)~~{x_1}^{G}\o S(x_3)_{1G}\o \tau(x_2,S(x_3)_2)=1_A\o 1_H\o 1_B\v_H(x).
 \end{eqnarray*}
 In this case, we call $H$ a $(G,\tau)$-Hopf algebra.
 \end{defi}

 \begin{rmk}\label{rmk:3.2a}(1) Let $G$ and $\tau$ be trivial, i.e., $x^G\o y_G=1_A\o xy$ and $\tau(x, y)=1_B\v(x)v(y)$ hold for all $x, y\in H$. Then $(I1)$ and $(I2)$ can ensure that $S$ is the antipode of $H$. Therefore in this case, the Hopf algebra $H$ is a $(G,\tau)$-Hopf algebra.\newline
 \indent{\phantom{\bf Remarks}}\quad (2) Let $\s: H\o H\longrightarrow A$ be linear map. When $G(x, x')=\s(x_1, x'_1)\o x_2x'_2$ in Definition \ref{de:3.2}, we call $H$ a {\it $(\s,\tau)$-Hopf algebra}.
 \end{rmk}

 \begin{pro} \label{pro:3.3}  Let $A, B$ be two Hopf algebras with antipode $S_A, S_B$, $H$ a coalgebra with an element $1_H$ such that $\D(1_H)=1_H\o 1_H$ and $S_H: H\lr H$ a linear map. Suppose that $A\star {_{R}^{G} H^{\tau}_{T}}\star B$ is a Brzezi\'{n}ski's two-sided crossed product bialgebra. Then $A\star {_{R}^{G} H^{\tau}_{T}}\star B$ is a Hopf algebra with antipode $\overline{S}$ defined by
 \begin{eqnarray}\label{eq:s}
 \overline{S}(a\o x\o b)=(1_A\o 1_H\o S_B(b))(1_A\o S_H(x)\o 1_B)(S_A(a)\o 1_H\o 1_B)
 \end{eqnarray}
 if and only if $H$ is a $(G,\tau)$-Hopf algebra.
 \end{pro}

 \begin{proof} $(\Longleftarrow)$  For all $a\in A, b\in B, x\in H$, and $R=r, T=t$, we have
 \begin{eqnarray*}
 &&(\overline{S}\ast id_{A\star{_{R}^{\ \s} H^{\tau}_{T}}\star B})(a\o x\o b)\\
 &=&S_A(a_1)_Ra_{2 r}{S_H(x_1)_{T R r 1}}^{G}\o x_{2 t 1 G}\o \tau(S_H(x_1)_{T R r 2},x_{2 t 2})S_B(b_1)_{T t}b_2\\
 &\stackrel{(LB2)}{=}& (S_A(a_1)a_2)_R{S_H(x_1)_{T R 1}}^{G}\o x_{2 t 1 G}\o \tau(S_H(x_1)_{T R 2},x_{2 t 2})S_B(b_1)_{T t}b_2\\
 &\stackrel{(BT3)}{=}&  {S_H(x_1)_1}^{G}\o x_{2 G T}\o S(b_1)_T\tau(S(x_1)_2,x_3)b_2\varepsilon(a)\\
 &\stackrel{(I1)}{=}& 1_A\o 1_{H T}\o S(b_1)_Tb_2\v(a)\v(x)\\
 &=&1_A\o 1_{H}\o 1_B\v(a)\v(x)\v(b)
 \end{eqnarray*}
 and
 \begin{eqnarray*}
 &&(id_{A\star{_{R}^{\ \s} H^{\tau}_{T}}\star B}\ast \overline{S})(a\o x\o b)\\
 &=&a_1S_A(a_2)_{R r}{x_{1 r 1}}^{G}\o S_H(x_2)_{T R t 1 G}\o \tau(x_{1 r 2},S_H(x_2)_{T R t 2})b_{1 t}S_B(b_2)_T\\
 &\stackrel{(BT4)}{=}& a_1S_A(a_2)_{R r}{x_{1 r 1}}^{G}\o S_H(x_2)_{R T t 1 G}\o \tau(x_{1 r 2},S_H(x_2)_{R T t 2})b_{1 t}S_B(b_2)_T\\
 &\stackrel{(RB2)}{=}& a_1S_A(a_2)_{R r}{x_{1 r 1}}^{G}\o S_H(x_2)_{R T 1 G}\o \tau(x_{1 r 2},S_H(x_2)_{R T 2})(b_{1}S_B(b_2))_T\\
 &=&a_1S_A(a_2)_{R r}{x_{1 r 1}}^{G}\o S_H(x_2)_{R 1 G}\o \tau(x_{1 r 2},S_H(x_2)_{R 2})\v(b)\\
 &\stackrel{(BT2)}{=}&a_1{x_1}^{G}S_A(a_2)_R\o S_H(x_3)_{1 G R}\o \tau(x_2,S_H(x_3)_2)\v(b)\\
 &\stackrel{(I2)}{=}&a_1S_A(a_2)_R\o 1_{H R}\o1_B\v(x)\v(b)\\
 &=&1_A\o 1_{H}\o 1_B\v(a)\v(x)\v(b).
 \end{eqnarray*}
 Thus, $\overline{S}$ is the antipode of $A\#{_{R}^{\ \s} H^{\tau}_{T}}\# B$.

 $(\Longrightarrow)$ Obvious.
 \end{proof}

 Next we only list two special cases. Other cases can be given similarly.

 \begin{cor} \label{cor:3.1a} Let $H$, $A$, $B$ be bialgebras. Let $\s: H\otimes H \longrightarrow A\otimes H$, $R: H\otimes A \longrightarrow A\otimes H$, $T: B\o H \lr H\o B$ and $\tau: H\otimes H \longrightarrow B$ be linear maps such that $\v_B(\tau(x,x'))=\v_H(x)\v_H(x')$. Then $A\#{_{R}^{\s} H^{\tau}_{T}}\# B$ equipped with the two-sided tensor product coalgebra $A\o H\o B$ becomes a bialgebra if and only if Eqs. $(G2)$-$(G4)$ and the following conditions hold:
 \begin{eqnarray*}
  &(J1)& \v_{A}(\s(x, x'))=\v_{H}(x)\v_{H}(x');\\
  &(J2)& {\s(x,x')}_{1}\o {\s(x,x')}_{2}={\s(x_{1},x'_{1})}\o \s(x_{2},x'_{2});\\
  &(J3)& {\tau(x,x')}_{1}\o {\tau(x,x')}_{2}={\tau(x_{1},x'_{1})}\o \tau(x_{2},x'_{2});\\
  &(J4)& {\tau(x_{1},x'_{1})}\o x_{2}x'_{2}={\tau(x_{2},x'_{2})}\o x_{1}x'_{1};\\
  &(J5)& {\s(x_{1},x'_{1})}\o x_{2}x'_{2}={\s(x_{2},x'_{2})}\o x_{1}x'_{1}\\
  &(J6)& \tau(x_{1},x'_{1})\otimes \sigma (x_{2},x'_{2}) =\tau(x_{2}, x'_{2})\otimes \sigma(x_{1},x'_{1}),
 \end{eqnarray*}
  where $x, x'\in H$.  In this case, we call this bialgebra {\it two-sided twisted crossed product bialgebra} and denoted by $A\star {_{R}^{\s} H^{\tau}_{T}}\star B$. Furthermore, $A\star {_{R}^{\s} H^{\tau}_{T}}\star B$ is a Hopf algebra with antipode Eq.(\ref{eq:s}) if and only if $H$ is a $(\s,\tau)$-Hopf algebra.
  \end{cor}

  \begin{proof} Note here Eqs.$(J2)$-$(J6)$ are equivalent to Eq.$(G5)$ for $G(x, x')=\s(x_1, x'_1)\o x_2x'_2$. The rest is obvious.  \end{proof}

 \begin{cor} \label{cor:3.1b} Let $H$ be a bialgebra. Let $\s: H\o H\ra A$ be a linear map, where $A$ is a bialgebra with a left $H$-weak action such that $\v_A(\s(x, x'))=\v_H(x)\v_H(x')$. Let $\tau: H\o H\ra B$ be a linear map, where $B$ is a bialgebra with a right $H$-weak action such that $\v_B(1_B)=1, \v_B(\tau(x,x'))=\v_H(x)\v_H(x')$. The two-sided crossed product $A\#^{\s} H~^{\tau}\# B$ equipped with the two-sided tensor product coalgebra $A\o H\o B$ becomes a bialgebra if and only if Eqs.$(J2)$-$(J6)$, $(D4)$, $(F3)$, $(F4)$  and the following condition holds:
 \begin{eqnarray*}
 &(P1) & x_{2} \otimes (x_{1}\triangleright a)=x_{1} \otimes (x_{2}\triangleright a),~~~ (b\triangleleft x_{2}) \otimes x_{1} =(b \triangleleft x_{1})\otimes x_{2},
 \end{eqnarray*}
 where $x\in H, a\in A$. In this case, we call this bialgebra {\it two-sided crossed product bialgebra} and denoted by $A\star {^{\s} H^{\tau}}\star B$. Furthermore, $A\star {^{\s} H^{\tau}}\star B$ is a Hopf algebra with antipode Eq.(\ref{eq:s}) if and only if $H$ is a $(\s,\tau)$-Hopf algebra.
  \end{cor}

  \begin{proof} It is straightforward by setting the left and right coactions are trivial in Theorem \ref{thm:dcb}.  \end{proof}

 \begin{ex} \label{ex:ss}
 Let $H=KC_2=K\{1, a\}$ and $A=B=KC_4=K\{1, x, x^2, x^3\}$ be group Hopf algebras (\cite{Ra12}), where $C_i$ is cyclic group with order $i$, $i=2,4$. Define the linear maps $\tl: H\o A\lr A$, $\tr: B\o H\lr B$, $\s: H\o H\lr A$ and $\tau:H\o H\lr B$ by
 \begin{eqnarray*}
 &&1_H\tl 1_A=1_A,~1_H\tl x=x,~1_H\tl x^2=x^2,~1_H\tl x^3=x^3,\\
 &&a\tl 1_A=1_A,~a\tl x=x^3,~a\tl x^2=x^2,~a\tl x^3=x;
 \end{eqnarray*}
 \begin{eqnarray*}
 &&1_B\tr 1_H=1_B,~x\tr 1_H=x,~x^2\tr 1_H=x^2,~x^3\tr 1_H=x^3,\\
 &&1_B\tr a=1_B,~x\tr a=x^3,~ x^2\tr a=x^2,~ x^3\tr a=x;
 \end{eqnarray*}
 and
 \begin{eqnarray*}
 &&\s(1_H, 1_H)=1_A,~\s(1_H, a)=1_A,~\s(a, 1_H)=1_A,~\s(a, a)=x;\\
 &&\tau(1_H, 1_H)=1_A,~\tau(1_H, a)=1_A,~\tau(a, 1_H)=1_A,~\tau(a, a)=x,
 \end{eqnarray*}
 then by \cite[Example 2.11]{AM08}, we know that $\tl$ is a weak action of $H$ on $A$ and $\tr$ is a weak action of $H$ on $B$. Thus we can get two-sided crossed product $A\# {^{\s} H^{\tau}}\# B$. Also by Corollary \ref{cor:3.1b}, two-sided crossed product $A\#^{\s} H~^{\tau}\# B$ equipped with the two-sided tensor product coalgebra $A\o H\o B$ becomes a bialgebra, since $H$ and $A, B$ are cocommucative. But unfortunately $\bar{S}$ defined by Eq.(\ref{eq:s}) is not an antipode of $A\star {^{\s} H^{\tau}}\star B$, because Eq.(I1) doesn't hold for $a\in H$.
 \end{ex}

\section{An extended version of Majid's double biproduct}
 Firstly, we recall the construction of double crossed biproduct in \cite{MLL}.
 \begin{lem}(\cite{MLL})\label{lem:dcb}
 Let $H$ be a bialgebra. Let $\s: H\o H\ra A$ be a linear map, where $A$ is a left $H$-comodule coalgebra such that $\v_A(1_A)=1, \v_A(\s(x, x'))=\v_H(x)\v_H(x')$, and an algebra with a left $H$-weak action. Let $\tau: H\o H\ra B$ be a linear map, where $B$ is a right $H$-comodule coalgebra such that $\v_B(1_B)=1, \v_B(\tau(x,x'))=\v_H(x)\v_H(x')$, and an algebra with a right $H$-weak action. The two-sided crossed product $A\#^{\s} H~^{\tau}\# B$ equipped with the two-sided smash coproduct $A\times H\times B$ becomes a bialgebra if and only if the following conditions hold:
 \begin{eqnarray*}
 &(C1) & \v_{A},~\v_{B}\hbox{~are algebra maps};\\
 &(C2)&\D_A(1_A)=1_A\o 1_A,~\D_B(1_B)=1_B\o 1_B;\\
 &(C3) & {1_A}_{-1}\o {1_A}_{0}=1_H\o 1_A,~{1_B}_{[0]}\o {1_B}_{[1]}=1_B\o 1_H;\\
 &(C4) & \v_{A}( x\tl a)=\v_{A}(a)\v_{H}(x),~\v_{B}(b\tr x)=\v_{B}(b)\v_{H}(x);\\
 &(C5) & a_{1}\o a_{2-1}x\o a_{20}=a_{1}\s (a_{2-11},x_{1})\o a_{2-12}x_{2}\o a_{20};\\
 &(C6) & b_{1[0]}\o xb_{1[1]}\o b_{2}=b_{1[0]}\o x_{1}b_{1[1]}\o \tau(x_{2},b_{1[2]})b_{2};\\
 &(C7) & (aa')_{1}\o (aa')_{2-1}\o (aa')_{20}\\
 &&~~~~~~~ =a_{1}(a_{2-11}\tl a'_{1})\s(a_{2-12}, a'_{2-11})\o a_{2-13}a'_{2-12}\o a_{20}a'_{20};\\
 &(C8) & (bb')_{1[0]}\o (bb')_{1[1]}\o (bb')_{2}=b_{1[0]}b'_{1[0]}\o b_{1[1]}b'_{1[1]}\o \tau(b_{1[2]},b'_{1[2]})(b_2\tr b'_{1[3]})b'_{2};\\
 &(C9) & \s (x_{1},x'_{1})_{1}\o \s (x_{1},x'_{1})_{2-1}x_{2}x'_{2}\o \s(x_{1},x'_{1})_{20}=\s (x_{1},x'_{1})\o x_{2}x'_{2}\o \s(x_{3},x'_{3});\\
 &(C10) & \tau (x_{2},x'_{2})_{1[0]}\o x_{1}x'_{1}\tau (x_{2},x'_{2})_{1[1]}\o \tau(x_{2},x'_{2})_{2}=\tau (x_{1},x'_{1})\o x_{2}x'_{2}\o \tau(x_{3},x'_{3});\\
 &(C11) & (x_{1}\tl a)_{1}\o ( x_{1}\tl a)_{2-1}x_{2}\o (x_{1}\tl a)_{20}\\
 &&~~~~~~~=(x_{1}\tl a_{1})\s (x_{2},a_{2-11})\o x_{3}a_{2-12}\o x_{4}\tl a_{20};\\
 &(C12) & (b\tr x'_{2} )_{1[0]}\o x'_{1}(b\tr x'_{2})_{1[1]}\o (b\tr x'_{2})_{2}=b_{1[0]}\tr x'_{1}\o b_{1[1]}x'_{2}\o \tau(b_{1[2]},x'_{3})(b_{2}\tr x'_{4});\\
 &(C13) & a_{-1}x_{1}a'_{-1}x'_{1}\o \tau(x_{4},x'_{3})(b_{[0]}\tr x'_{4})b'_{[0]}\o a_{0}(x_{2}\tl a'_{0})\s (x_{3},x'_{2})\o x_{5}b_{[1]}x'_{5}b'_{[1]}\\
 &&~~~~~~~=a_{-11}x_{1}a'_{-11}x'_{1}\o \tau(a_{-12}x_{2},a'_{-12}x'_{2})(b_{[0]}\tr a'_{-13}x'_{3})b'_{[0]}\\
 &&~~~ ~~~~~~~~~~~~~~~~~~\o a_{0}(x_{3}b_{[1]}\tl a'_{0})\s (x_{4}b_{[2]},x'_{4}b'_{[1]})\o x_{5}b_{[3]}x'_{5}b'_{[2]}.
 \end{eqnarray*}
 In this case, we call this bialgebra the {\it double crossed biproduct} denoted by $A\star^{\s} H~{^{\tau}\star} B$.
 \end{lem}

 Next we give an improved version of the above result.

 \begin{thm}\label{thm:dcb}
 Under the assumption of Lemma \ref{lem:dcb}. The two-sided crossed product $A\#^{\s} H~^{\tau}\# B$ equipped with the two-sided smash coproduct $A\times H\times B$ becomes a bialgebra if and only if $(C1)-(C6)$ and the following conditions hold:
 \begin{eqnarray*}
 &(D7) &(aa')_{1}\otimes (aa')_{2}=a_{1}(a_{2-11}\triangleright a'_{1})\sigma(a_{2-12},a'_{2-1})\otimes a_{20}a'_{20};\\
 &(D8) & (aa')_{-1}\otimes (aa')_{0}=a_{-1}a'_{-1}\otimes a_{0}a'_{0};\\
 &(D9) &(bb')_{1}\otimes (bb')_{2}=b_{1[0]}b'_{1[0]}\otimes \tau (b_{1[1]},b'_{1[1]})(b_{2}\triangleleft b'_{1[2]})b'_{2};\\
 &(D10) &(bb')_{[0]}\otimes (bb')_{[1]}=b_{[0]}b'_{[0]}\otimes b_{[1]}b'_{[1]};\\
 &(D11) &\sigma(x_{1},x'_{1})_{-1}x_{2}x'_{2}\otimes \sigma(x_{1},x'_{1})_{0}=x_{1}x'_{1}\otimes \sigma(x_{2},x'_{2});\\
 &(D12) &\sigma(x,x')_{1}\otimes \sigma(x,x')_{2}=\sigma(x_{1},x'_{1})\otimes \sigma(x_{2},x'_{2});\\
 &(D13) &\tau(x_{2},x'_{2})_{[0]}\otimes x_{1}x'_{1}\tau(x_{2},x'_{2})_{[1]}=\tau(x_{1},x'_{1})\otimes x_{2}x'_{2};\\
 &(D14) &\tau(x,x')_{1}\otimes \tau(x,x')_{2}=\tau(x_{1},x'_{1})\otimes \tau(x_{2},x'_{2});\\
 &(D15) &(x\triangleright a)_{1}\otimes (x\triangleright a)_{2}=(x_{1}\triangleright a_{1})\sigma(x_{2},a_{2-1})\otimes x_{3}\triangleright a_{20};\\
 &(D16) &(x_{1}\triangleright a)_{-1}x_{2}\otimes (x_{1}\triangleright a)_{0}=x_{1}a_{-1}\otimes x_{2}\triangleright a_{0};\\
 &(D17) &(b\triangleleft x_{2})_{[0]}\otimes x_{1}(b\triangleleft x_{2})_{[1]}=b_{[0]}\triangleleft x_{1}\otimes b_{[1]}x_{2};\\
 &(D18) &(b\triangleleft x)_{1}\otimes (b\triangleleft x)_{2}=b_{1[0]}\triangleleft x_{1}\otimes \tau(b_{1[1]},x_{2})(b_{2}\triangleleft x_{3});\\
 &(D19) &a_{-1}x\otimes 1_{B} \otimes a_{0}=a_{-11}x_{1}\otimes \tau (a_{-12},x_{2})\otimes a_{0} ;\\
 &(D20) &b_{[0]}\otimes 1_{A}\otimes xb_{[1]}=b_{[0]}\otimes \sigma (x_{1},b_{[1]1})\otimes x_{2}b_{[1]2} ;\\
 &(D21) &b_{[1]}\tl a_0\o b_{[0]}\tr a_{-1}=a\o b ;\\
 &(D22) &\tau(x_{3},x'_{2})(b\triangleleft x'_{3})b'\otimes a(x_{1}\triangleright a')\sigma (x_{2},x'_{1})\\
 && =\tau(a_{-1}x_{1},a'_{-1}x'_{1})(b_{[0]} \triangleleft x'_{2})b'_{[0]}\otimes a_{0}(x_{2}\triangleright a'_{0})\sigma(x_{3}b_{[1]},x'_{3}b'_{[1]}) .
 \end{eqnarray*}
 \end{thm}

 \begin{proof} By Lemma \ref{lem:dcb}, we only need to prove that the conditions $(C7)$-$(C13)$ are equivalent to $(D7)$-$(D22)$.
 \begin{enumerate}

 \item Step 1. Firstly, apply $id\otimes \varepsilon \otimes id$ to Eq.$(C7)$, we can get Eq.$(D7)$; apply $\varepsilon \otimes id \otimes id$ to Eq.$(C7)$, then Eq.$(D8)$ holds. On the other hand, by Eqs.$(D7)$ and $(D8)$, we have
 \begin{eqnarray*}
 (aa')_{1}\otimes (aa')_{2-1}\otimes (aa')_{20}
 &\stackrel{(D7)}{=} &a_{1}(a_{2-11}\triangleright a'_{1})\sigma (a_{2-12},a'_{2-1})\otimes (a_{20}a'_{20})_{-1}\otimes (a_{20}a'_{20})_{0}\\
 &\stackrel{(D8)}{=} &a_{1}(a_{2-11}\triangleright a'_{1})\sigma (a_{2-12},a'_{2-1})\otimes a_{20-1}a'_{20-1}\otimes a_{200}a'_{200}\\
 &=&a_{1}(a_{2-11}\triangleright a'_{1})\sigma(a_{2-12},a'_{2-11})\otimes a_{2-13}a'_{2-12}\otimes a_{20}a'_{20},
 \end{eqnarray*}
 i.e., Eq.$(C7)$ holds. Thus Eq.$(C7)$ $\Leftrightarrow$ Eqs.$(D7)$ and $(D8)$.

 Similarly, we can obtain that Eq.$(C8)$ $\Leftrightarrow$ Eqs.$(D9)$ and $(D10)$; Eq.$(C9)$ $\Leftrightarrow$ Eqs.$(D11)$ and $(D12)$; Eq.$(C10)$ $\Leftrightarrow$ Eqs.$(D13)$ and $(D14)$; Eq.$(C11)$ $\Leftrightarrow$ Eqs.$(D15)$ and $(D16)$; Eq.$(C12)$ $\Leftrightarrow$ Eqs.$(D17)$ and $(D18)$.

 \item Step 2. Secondly, let $x=1, a'=1, b=b'=1$ in Eq.$(C13)$, we have
 \begin{eqnarray}\label{eq:c13-1}
 a_{-1}x'_{1}\otimes 1_{B}\otimes a_{0}\otimes x'_{2}=a_{-11}x'_{1}\otimes \tau (a_{-12},x'_{2})\otimes a_{0}\otimes x'_{3}
  \end{eqnarray}
 Using $id\otimes id\otimes id\otimes \varepsilon$ to Eq.$(\ref{eq:c13-1})$, one can get Eq.$(D19)$. Applying $\v\o id\o id$ to Eq.(D19), we have
 \begin{eqnarray}\label{eq:d19-1}
 a\o 1_B\v(x)=a_{0}\o \tau(a_{-1}, x).
 \end{eqnarray}

 Set $a=a'=1, b=1, x'=1$ in Eq.$(C13)$, we have
 \begin{eqnarray}\label{eq:c13-2}
 x_{1}\otimes b'_{[0]}\otimes 1_{A}\otimes x_{2}b'_{[1]}=x_{1}\otimes b'_{[0]}\otimes \sigma(x_{2},b'_{[1]})\otimes x_{3}b'_{[2]}
  \end{eqnarray}
 Applying $\varepsilon \otimes id\otimes id\otimes id$ to Eq.$(\ref{eq:c13-2})$, one can get Eq.$(D20)$. Applying $id\o id\o \v$ to Eq.(D20), we have
 \begin{eqnarray}\label{eq:d20-1}
 b\o 1_A\v(x)=b_{[0]}\o \s(x, b_{[1]}).
 \end{eqnarray}

 Setting $a=1, b'=1, x=x'=1$ in Eq.$(C13)$, we have
 \begin{eqnarray}\label{eq:c13-4}
 a'_{-1}\otimes b_{[0]}\otimes a'_{0}\otimes b_{[1]}=a'_{-11}\otimes (b_{[0]}\triangleleft  a'_{-12})\otimes (b_{[1]}\triangleright a'_{0})\otimes b_{[2]}.
  \end{eqnarray}
 Applying $\varepsilon \otimes id\otimes id\otimes \varepsilon$ to Eq.$(\ref{eq:c13-4})$, we have Eq.$(D21)$.

 Applying $\varepsilon \otimes id\otimes id\otimes \v$ to Eq.$(C13)$, we have
 \begin{eqnarray}
 &&\tau(x_{3},x'_{2})(b\triangleleft x'_{3})b'\otimes a(x_{1}\triangleright a')\sigma (x_{2},x'_{1})\label{eq:c13-3}\\
 &&~~~~=\tau(a_{-1}x_{1},a'_{-11}x'_{1})(b_{[0]}\triangleleft a'_{-12}x'_{2})b'_{[0]}\otimes a_{0}(x_{2}b_{[1]1}\triangleright a'_{0})\sigma (x_{3}b_{[1]2},x'_{3}b'_{[1]}) .\nonumber
  \end{eqnarray}
 While
 \begin{eqnarray*}
 \hbox{RHS~ of~Eq}.(\ref{eq:c13-3})
 &\stackrel{(D19)(D20)}{=} &\tau(a_{-1}x_{1},a'_{-11}x'_{1})(b_{[0]}\triangleleft a'_{-12}x'_{2})\tau(a'_{-13},x'_{3})b'_{[0]}\\
 &&\otimes a_{0}\sigma(x_{2},b_{[1]1})(x_{3}b_{[1]2}\triangleright a'_{0})\sigma (x_{4}b_{[1]3},x'_{4}b'_{[1]})\\
 &\stackrel{(LC3)(RC3)}{=}&\tau(a_{-1}x_{1},a'_{-11}x'_{1})\tau(a'_{-12},x'_{2})
 [(b_{[0]}\triangleleft a'_{-13})\triangleleft x'_{3}]b'_{[0]}\\
 &&\otimes a_{0}(x_{2}\triangleright (b_{[1]1}\triangleright a'_{0}))\sigma(x_{3},b_{[1]2})\sigma (x_{4}b_{[1]3},x'_{4}b'_{[1]})\\
 &\stackrel{ }{=}&\tau(a_{-1}x_{1},a'_{-11}x'_{1})\tau(a'_{-12},x'_{2})
 ((b_{[0][0]}\triangleleft a'_{0-1})\triangleleft x'_{3})b'_{[0]}\\
 &&\otimes a_{0}(x_{2}\triangleright (b_{[0][1]}\triangleright a'_{00}))\sigma(x_{3},b_{[1]1})\sigma (x_{4}b_{[1]2},x'_{4}b'_{[1]})\\
 &\stackrel{(D21)}{=}&\tau(a_{-1}x_{1},a'_{-11}x'_{1})\tau(a'_{-12},x'_{2})
 (b_{[0]}\triangleleft x'_{3})b'_{[0]}\\
 &&\otimes a_{0}(x_{2}\triangleright a'_{0})\sigma(x_{3},b_{[1]1})\sigma (x_{4}b_{[1]2},x'_{4}b'_{[1]})\\
 &\stackrel{(D19)(D20)}{=} &\tau(a_{-1}x_{1},a'_{-1}x'_{1})(b_{[0]}\triangleleft x'_{2})b'_{[0]} \otimes a_{0}(x_{2}\triangleright a'_{0})\sigma(x_{3}b_{[1]},x'_{3}b'_{[1]})\\
 &=&\hbox{RHS~ of~Eq}.(D22)
 \end{eqnarray*}
 and LHS of Eq.(\ref{eq:c13-3}) is exactly LHS of Eq.$(D22)$, so we obtain Eq.$(D22)$.

 On the other hand, we have
 \begin{eqnarray*}
 \hbox{RHS~ of~Eq}.(C13)
 &\stackrel{(D20)}{=} &a_{-11}x_{1}a'_{-11}x'_{1}\otimes \tau (a_{-12}x_{2},a'_{-12}x'_{2})(b_{[0]}\triangleleft a'_{-13}x'_{3})b'_{[0]}\\
 &&\otimes a_{0}\sigma (x_{3},b_{[1]1})(x_{4}b_{[1]2}\triangleright a'_{0})\sigma (x_{5}b_{[1]3},x'_{4}b'_{[1]1})\otimes x_{6}b_{[1]4}x'_{5}b'_{[1]2}\\
 &\stackrel{(D19)}{=} &a_{-11}x_{1}a'_{-11}x'_{1}\otimes \tau (a_{-12}x_{2},a'_{-12}x'_{2})(b_{[0]}\triangleleft a'_{-13}x'_{3})\tau (a'_{-14},x'_{4})b'_{[0]}\\
 &&\otimes a_{0}\sigma (x_{3},b_{[1]1})(x_{4}b_{[1]2}\triangleright a'_{0})\sigma (x_{5}b_{[1]3},x'_{5}b'_{[1]1})\otimes x_{6}b_{[1]4}x'_{6}b'_{[1]2}\\
 &\stackrel{(LC3)(RC3)}{=} &a_{-11}x_{1}a'_{-11}x'_{1}\otimes \tau (a_{-12}x_{2},a'_{-12}x'_{2})\tau (a'_{-13},x'_{3})((b_{[0]}\triangleleft a'_{-14})\triangleleft x'_{4})b'_{[0]}\\
 &&\otimes a_{0}(x_{3}\triangleright (b_{[1]1}\triangleright a'_{0}))\sigma (x_{4},b_{[1]2})\sigma (x_{5}b_{[1]3},x'_{5}b'_{[1]1})\otimes x_{6}b_{[1]4}x'_{6}b'_{[1]2}\\
 &\stackrel{}{=} &a_{-11}x_{1}a'_{-11}x'_{1}\otimes \tau (a_{-12}x_{2},a'_{-12}x'_{2})\tau (a'_{-13},x'_{3})((b_{[0][0]}\triangleleft a'_{0-1})\triangleleft x'_{4})b'_{[0]}\\
 &&\otimes a_{0}(x_{3}\triangleright (b_{[0][1]}\triangleright a'_{00}))\sigma (x_{4},b_{[1]1})\sigma (x_{5}b_{[1]2},x'_{5}b'_{[1]1})\otimes x_{6}b_{[1]3}x'_{6}b'_{[1]2}\\
 &\stackrel{(D21)}{=} &a_{-11}x_{1}a'_{-11}x'_{1}\otimes \tau (a_{-12}x_{2},a'_{-12}x'_{2})\tau (a'_{-13},x'_{3})(b_{[0]}\triangleleft x'_{4})b'_{[0]}\\
 &&\otimes a_{0}(x_{3}\triangleright a'_{0})\sigma (x_{4},b_{[1]1})\sigma (x_{5}b_{[1]2},x'_{5}b'_{[1]1})\otimes x_{6}b_{[1]3}x'_{6}b'_{[1]2}\\
 &\stackrel{}{=} &a_{-11}x_{1}a'_{-11}x'_{1}\otimes \tau (a_{-12}x_{2},a'_{-12}x'_{2})\tau (a'_{0-1},x'_{3})(b_{[0][0]}\triangleleft x'_{4})b'_{[0]}\\
 &&\otimes a_{0}(x_{3}\triangleright a'_{00})\sigma (x_{4},b_{[0][1]})\sigma (x_{5}b_{[1]1},x'_{5}b'_{[1]1})\otimes x_{6}b_{[1]2}x'_{6}b'_{[1]2}\\
 &\stackrel{(\ref{eq:d19-1})(\ref{eq:d20-1})}{=} &a_{-11}x_{1}a'_{-11}x'_{1}\otimes \tau (a_{-12}x_{2},a'_{-12}x'_{2})(b_{[0]}\triangleleft x'_{3})b'_{[0]}\\
 &&\otimes a_{0}(x_{3}\triangleright a'_{0})\sigma (x_{4}b_{[1]1},x'_{4}b'_{[1]1})\otimes x_{5}b_{[1]2}x'_{5}b'_{[1]2}\\
 &\stackrel{}{=} &a_{-1}x_{1}a'_{-11}x'_{1}\otimes \tau (a_{0-1}x_{2},a'_{-12}x'_{2})(b_{[0]}\triangleleft x'_{3})b'_{[0][0]}\\
 &&\otimes a_{00}(x_{3}\triangleright a'_{0})\sigma (x_{4}b_{[1]1},x'_{4}b'_{[0][1]})\otimes x_{5}b_{[1]2}x'_{5}b'_{[1]}\\
 &\stackrel{}{=} &a_{-1}x_{1}a'_{-1}x'_{1}\otimes \tau (a_{0-1}x_{2},a'_{0-1}x'_{2})(b_{[0][0]}\triangleleft x'_{3})b'_{[0][0]}\\
 &&\otimes a_{00}(x_{3}\triangleright a'_{00})\sigma (x_{4}b_{[0][1]},x'_{4}b'_{[0][1]})\otimes x_{5}b_{[1]}x'_{5}b'_{[1]}\\
 &\stackrel{(D22)}{=}&\hbox{LHS~ of~Eq}.(C13).
  \end{eqnarray*}
 Thus Eq.$(C13)$ $\Leftrightarrow$ Eqs.$(D19)$-$(D22)$.  These finish the proof.
 \end{enumerate}\end{proof}

 \begin{rmk}\label{rmk:dcb}
 (1) Eq.$(D21)$ is exactly Eq.(\ref{eq:DB}) in the Majid's double biproduct. \newline
 \indent{\phantom{\bf Remarks}}\quad (2) The first half parts in Eqs.$(C1)$-$(C4)$ and Eqs.$(D5)$, $(D7)$, $(D8)$, $(D11)$, $(D12)$, $(D15)$, $(D16)$ are the conditions $A_1)$-$A_9)$ and twisted comodule cocycle in \cite[Theorem 2.5]{WJZ}. \newline
 \indent{\phantom{\bf Remarks}}\quad (3) The second half parts in Eqs.$(C1)$-$(C4)$ and Eqs.$(D6)$, $(D9)$, $(D10)$, $(D13)$, $(D14)$, $(D17)$, $(D18)$ are the conditions $(G1)$-$(G9)$ in \cite[Remark 2.3]{MLL} and $(C6)$ in \cite[Theorem 2.1]{MLL}. These are the necessary and sufficient conditions for the right crossed product $H~^{\tau}\# B$ equipped with the right smash coproduct $H\times B$ becomes a bialgebra.

 \end{rmk}

 \begin{pro}\label{pro:dcb-1} Let $H$ be a bialgebra. Let $\s: H\o H\ra A$ be a linear map, where $A$ is a left $H$-comodule coalgebra such that $\v_A(1_A)=1$, and an algebra with a left $H$-weak action. Let $B$ be a right $H$-module algebra and also a comodule coalgebra such that $\v_B(1_B)=1$. The one-sided crossed product $A\#^{\s} H\# B$ equipped with the two-sided smash coproduct $A\times H\times B$ becomes a bialgebra if and only if Eqs.$(C1)$-$(C5)$, $(D7)$, $(D8)$, $(D10)$-$(D12)$, $(D15)$-$(D17)$, $(D20)$, $(D21)$ and the following conditions hold:
 \begin{eqnarray*}
 &(E1) &(bb')_{1}\otimes (bb')_{2}=b_{1}b'_{1[0]}\otimes (b_{2}\triangleleft b'_{1[1]})b'_{2};\\
 &(E2) &(b\triangleleft x)_{1}\otimes (b\triangleleft x)_{2}=(b_{1}\triangleleft x_{1})\otimes (b_{2}\triangleleft x_{2});\\
 &(E3) &(b_{[0]}\triangleleft x'_{2})b'\otimes  (x_{1}\triangleright a)\sigma (x_{2},x'_{1})=(b_{[0]}\triangleleft x'_{1})b'_{[0]}\otimes (x_{1}\triangleright a)\sigma(x_{2}b_{[1]},x'_{2}b'_{[1]}) .
 \end{eqnarray*}
 \end{pro}

 \begin{proof} Let $\tau$ be trivial, i.e., $\tau(x,y)=\v(x)\v(y)1_B$ in Theorem \ref{thm:dcb}.  \end{proof}

 \begin{rmk}\label{rmk:dcb-1}
 (1) By Eq.$(D21)$, Eq.(\ref{eq:DB}) in the Majid's double biproduct is not only sufficient but also necessary. \newline
 \indent{\phantom{\bf Remarks}}\quad (2) The conditions in Proposition \ref{pro:dcb-1} are equivalent to the ones in \cite[Theorem 3.2]{MJS}.

 \begin{proof}  By Eq.$(D22)$, we have
 \begin{eqnarray}\label{eq:d22-1}
 (b_{[0]}\triangleleft x'_{1})b'_{[0]}\otimes a(x_{1}\triangleright a')\sigma(x_{2}b_{[1]},x'_{2}b'_{[1]})=(b\triangleleft x'_{2})b'\otimes a(x_{1}\triangleright a')\sigma(x_{2},x'_{1}).
  \end{eqnarray}
 Setting $a=1=a', b'=1$ in Eq.(\ref{eq:d22-1}), we have
 \begin{eqnarray}\label{eq:d22-2}
 (b_{[0]}\triangleleft x'_{1})\otimes \sigma (xb_{[1]},x'_{2})=(b\triangleleft x'_{2})\otimes \sigma(x,x'_{1}).
 \end{eqnarray}
 Setting $a=1=a', b=1$ in Eq.(\ref{eq:d22-1}), we get
 \begin{eqnarray}\label{eq:d22-3}
 b'_{[0]}\otimes \sigma (x,x'b'_{[1]})=b'\otimes \sigma(x,x').
 \end{eqnarray}
 On the other hand, we have
 \begin{eqnarray*}
 &&(b_{[0]}\triangleleft x'_{1})b'_{[0]}\otimes a(x_{1}\triangleright a')\sigma(x_{2}b_{[1]},x'_{2}b'_{[1]})\\
 &\stackrel{(\ref{eq:d22-3})}{=} &(b_{[0]}\triangleleft x'_{1})b'\otimes a(x_{1}\triangleright a')\sigma(x_{2}b_{[1]},x'_{2})\\
 &\stackrel{(\ref{eq:d22-2})}{=} &(b\triangleleft x'_{2})b'\otimes a(x_{1}\triangleright a')\sigma(x_{2},x'_{1}).
  \end{eqnarray*}
 Therefore Eq.(\ref{eq:d22-1})$\Leftrightarrow$ Eqs.(\ref{eq:d22-2}) and (\ref{eq:d22-3}).

 $(\Longrightarrow)$ We only check that Eq.$(C14)$ in \cite[Theorem 3.2]{MJS} holds. In fact, one can compute as follows.
 \begin{eqnarray*}
 \hbox{LHS~ of~Eq}.(C14)
 &\stackrel{ }{=}&(b_{[0]}\triangleleft x'_{2})b'_{[0]}\otimes (x_{1}\triangleright a')\sigma(x_{2},x'_{1})\otimes x_3b_{[1]}x'_3b'_{[1]}\\
 &\stackrel{(\ref{eq:d22-2})}{=} &(b_{[0][0]}\triangleleft x'_{1})b'_{[0]}\otimes (x_{1}\triangleright a')\sigma(x_{2}b_{[0][1]},x'_{2})\otimes x_3b_{[1]}x'_3b'_{[1]}\\
 &\stackrel{(\ref{eq:d22-3})}{=} &(b_{[0][0]}\triangleleft x'_{1})b'_{[0][0]}\otimes (x_{1}\triangleright a')\sigma(x_{2}b_{[0][1]},x'_{2}b'_{[0][1]})\otimes  x_3b_{[1]}x'_3b'_{[1]}\\
 &\stackrel{ }{=}&(b_{[0]}\triangleleft x'_{1})b'_{[0]}\otimes (x_{1}\triangleright a')\sigma(x_{2}b_{[1]1},x'_{2}b'_{[1]1})\otimes   x_3b_{[1]2}x'_3b'_{[1]2}\\
 &\stackrel{(\ref{eq:d20-1})}{=} &(b_{[0][0]}\triangleleft x'_{1})b'_{[0]}\otimes (x_{1}\triangleright a')\sigma(x_{2},b_{[0][1]})\sigma(x_{3}b_{[1]1},x'_{2}b'_{[1]1})\otimes x_4b_{[1]2}x'_3b'_{[1]2}\\
 &\stackrel{ }{=} &(b_{[0]}\triangleleft x'_{1})b'_{[0]}\otimes (x_{1}\triangleright a')\sigma(x_{2},b_{[1]1})\sigma(x_{3}b_{[1]2},x'_{2}b'_{[1]1})\otimes x_4b_{[1]3}x'_3b'_{[1]2}\\
 &\stackrel{(D21)}{=} &(b_{[0][0]}\triangleleft a'_{-1}x'_{1})b'_{[0]}\otimes [x_{1}\triangleright (b_{[0][1]}\triangleright a'_{0})]\sigma(x_{2},b_{[1]1})\sigma(x_{3}b_{[1]2},x'_{2}b'_{[1]1})\\
 &&\otimes x_4b_{[1]3}x'_3b'_{[1]2}\\
 &\stackrel{ }{=} &(b_{[0]}\triangleleft a'_{-1}x'_{1})b'_{[0]}\otimes [x_{1}\triangleright (b_{[1]1}\triangleright a'_{0})]\sigma(x_{2},b_{[1]2})\sigma(x_{3}b_{[1]3},x'_{2}b'_{[1]1})\\
 &&\otimes x_4b_{[1]4}x'_3b'_{[1]2}\\
 &\stackrel{(LC3)}{=} &(b_{[0]}\triangleleft a'_{-1}x'_{1})b'_{[0]}\otimes \sigma(x_{1},b_{[1]1})(x_{2}b_{[1]2}\triangleright a'_{0})\sigma(x_{3}b_{[1]3},x'_{2}b'_{[1]1})\\
 &&\otimes x_4b_{[1]4}x'_3b'_{[1]2}\\
 &\stackrel{(D20)}{=} &(b_{[0]}\triangleleft a'_{-1}x'_{1})b'_{[0]}\otimes (x_{1}b_{[1]1}\triangleright a'_{0})\sigma(x_{2}b_{[1]2},x'_{2}b'_{[1]1})\otimes x_3b_{[1]3}x'_3b'_{[1]2}\\
 &=&\hbox{RHS~ of~Eq}.(C14)
 \end{eqnarray*}

 $(\Longleftarrow)$ Let $b=1, a'=1, x'=1$ in Eq.$(C14)$, then we can get Eq.$(D20)$; Applying $id\o id\o \v$ to Eq.$(C14)$, one can obtain
 \begin{eqnarray}\label{eq:c14-1}
 (b\triangleleft x'_{2})b'\otimes (x_{1}\triangleright a')\sigma(x_{2},x'_{1})=(b_{[0]}\triangleleft a'_{-1}x'_{1})b'_{[0]}\otimes (x_{1}b_{[1]1}\triangleright a'_{0})\sigma(x_{2}b_{[1]2},x'_{2}b'_{[1]}).
 \end{eqnarray}
 Let $a'=1, b'=1$ in Eq.(\ref{eq:c14-1}), we get Eq.(\ref{eq:d22-2}). Likewise, we can obtain Eqs.(\ref{eq:d22-3}) and $(D21)$.

 Up to now, we have checked that Eqs.$(D20), (D21), (E3)$ $\Leftrightarrow$ Eq.$(C14)$ in \cite[Theorem 3.2]{MJS}. And the rest is obvious.  \end{proof}
 \end{rmk}

 \begin{cor}\label{cor:dcb-2} Let $H$ be a bialgebra. Let $A$ be a left $H$-comodule coalgebra such that $\v_A(1_A)=1$, and an algebra with a left $H$-action $\tl$. Let $B$ be a right $H$-comodule coalgebra such that $\v_B(1_B)=1$, and an algebra with a right $H$-action $\tr$. The two-sided smash product $A\# H\# B$ equipped with the two-sided smash coproduct $A\times H\times B$ becomes a bialgebra if and only if Eqs.$(C1)$-$(C4)$, $(D8)$,$(D10)$,$(D16)$,$(D17)$,$(D21)$ and the following conditions hold:
 \begin{eqnarray*}
 &(F1) &(aa')_{1}\otimes (aa')_{2}=a_{1}(a_{2-1}\triangleright a'_{1})\otimes a_{20}a'_{2};\\
 &(F2) &(bb')_{1}\otimes (bb')_{2}=b_{1}b'_{1[0]}\otimes (b_{2}\triangleleft b'_{1[1]})b'_{2};\\
 &(F3) &(x\triangleright a)_{1}\otimes (x\triangleright a)_{2}=(x_{1}\triangleright a_{1})\otimes (x_{2}\triangleright a_{2});\\
 &(F4) &(b\triangleleft x)_{1}\otimes (b\triangleleft x)_{2}=(b_{1}\triangleleft x_{1})\otimes (b_{2}\triangleleft x_{2}).
 \end{eqnarray*}
 \end{cor}

 \begin{proof}  Let $\s$ and $\tau$ be trivial, i.e., $\s(x,y)=\v(x)\v(y)1_A$ and $\tau(x,y)=\v(x)\v(y)1_B$ in Theorem \ref{thm:dcb}.       \end{proof}

 \begin{rmk}\label{rmk:dcb-2}
 (1) By Eq.$(D21)$, Eq.(\ref{eq:DB}) in the Majid's double biproduct is not only sufficient but also necessary. \newline
 \indent{\phantom{\bf Remarks}}\quad (2) The first half parts in Eqs.$(C1)$-$(C4)$ and Eqs.$(F1)$, $(D8)$,$(F3)$, $(D16)$ and the assumption of left module algebra and left comodule coalgebra are equivalent to $A$ is a bialgebra in ${}^H_H{\mathbb{YD}}$. Similarly, the second half parts in Eqs.$(C1)$-$(C4)$ and Eqs.$(F2)$, $(D10)$, $(D17)$, $(F4)$ and the assumption of right module algebra and right comodule coalgebra are equivalent to $B$ is a bialgebra in ${\mathbb{YD}}{}^H_H$.
 \end{rmk}

 \begin{cor}(\cite{WJZ})\label{cor:dcb-3}
 Let $H$ be a bialgebra. Let $\s: H\o H\ra A$ be a linear map, where $A$ is a left $H$-comodule coalgebra, and an algebra with a left $H$-weak action. The left crossed product $A\#^{\s} H$ equipped with the left smash coproduct $A\times H$ becomes a bialgebra if and only if the first half parts in Eqs.$(C1)$-$(C4)$ and Eqs.$(D5)$, $(D7)$, $(D8)$, $(D11)$, $(D12)$, $(D15)$, $(D16)$ hold.
 \end{cor}

 \begin{proof} Let $B=K$ in Theorem \ref{thm:dcb}. \end{proof}

 \begin{rmk}\label{rmk:dcb-3}
 (1) Note that condition $(D5)$ is exactly the condition of the twisted comodule cocycle in \cite{WJZ}. Although this condition is the only prerequisite for the bialgebra $A\star^\s H$ in \cite{WJZ}, here it is just one of the necessary conditions. \newline
 \indent{\phantom{\bf Remarks}}\quad (2) The conditions in Corollary \ref{cor:dcb-3} are simpler than the ones in \cite[Corollary 3.3]{MJS}, but obvious they are equivalent.
 \end{rmk}

 \begin{cor} \label{cor:dcb-4}
 Let $H$ be a bialgebra. Let $\tau: H\o H\ra B$ be a linear map, where $B$ is a right $H$-comodule coalgebra, and an algebra with a right $H$-weak action. The right crossed product $H~^{\tau}\# B$ equipped with the right smash coproduct $H\times B$ becomes a bialgebra if and only if the second half parts in Eqs.$(C1)$-$(C4)$ and Eqs.$(D6)$, $(D9)$, $(D10)$, $(D13)$, $(D14)$, $(D17)$, $(D18)$ hold. In this case, we denote this bialgebra by $H~^\tau\star B$.
 \end{cor}

 \begin{proof} Let $A=K$ in Theorem \ref{thm:dcb}. \end{proof}

 \begin{rmk}\label{rmk:dcb-4}
 The conditions in Corollary \ref{cor:dcb-4} are simpler than the ones in \cite[Proposition 2.2]{MLL}, but obvious they are equivalent.
 \end{rmk}

\section{Conclusion} We end this paper by three questions:
 \begin{question} \label{q:1} Majid realized a categorical interpretation of Radford's biproduct (\cite{Maj1,Maj2}): $A\star H$ is a Radford's biproduct if and only if $A$ is a bialgebra in the Yetter-Drinfel'd category ${}^H_H{\mathbb{YD}}$. This ensure that $A\star H$ can play a central role in the classification of finite-dimensional pointed Hopf algebras. By Theorem \ref{thm:dcb} and Remark \ref{rmk:dcb}, we know that the first half parts in Eqs.$(C1)$-$(C4)$ and Eqs.$(D5)$, $(D7)$, $(D8)$, $(D11)$, $(D12)$, $(D15)$, $(D16)$ are the necessary and sufficient conditions for $A\star^\s H$ to be bialgebra. These conditions correspond to the conditions for Radford's biproduct one by one given in \cite{Ra}. Here we note that Eq.$(D16)$ is compatible condition for Yetter-Drinfeld module. Then whether there exists an appropriate category ${}^H_H{\mathbb{YD}}^\s$ such that the conditions above for $A\star^\s H$ hold if and only if $A$ is a bialgebra in the category ${}^H_H{\mathbb{YD}}^\s$.
 \end{question}

 \begin{question} \label{q:3} In Theorem \ref{thm:dcb}, we give the necessary and sufficient conditions for the two-sided crossed product $A\#^{\s} H~^{\tau}\# B$ equipped with the two-sided smash coproduct $A\times H\times B$ becomes a bialgebra. In Theorem \ref{thm:2.3}, a more general two-sided crossed product algebra $A\#{_{R}^{G} H^{\tau}_{T}}\# B$ is constructed. Replaced $A\#^{\s} H~^{\tau}\# B$ by $A\#{_{R}^{G} H^{\tau}_{T}}\# B$, what are the new conditions for the new bialgebra?
 \end{question}

 \begin{question} \label{q:2} In Theorem \ref{thm:dcb}, we only obtain the bialgebra structure $A\star^{\s} H~{^{\tau}\star} B$. It is natural to ask when $A\star^{\s} H~{^{\tau}\star} B$ is a Hopf algebra, i.e., how to construct the antipode for $A\star^{\s} H~{^{\tau}\star} B$. Although we provide the antipode in Proposition \ref{pro:3.3} for $A\star {_{R}^{F} H^{\tau}_{T}}\star B$, unfortunately, the coalgebra structure of $A\star {_{R}^{F} H^{\tau}_{T}}\star B$ is just the two-sided tensor product coalgebra. Recently, the authors in \cite{SW} give the antipode for a class of Majid's double biproduct by (co)triangular Hopf quasigroups.
 \end{question}

 {\bf Acknowledgments.} The authors are deeply indebted to the referee for his/her very useful suggestions and some improvements to the original manuscript. This work was partially supported by Natural Science Foundation of Henan Province (No. 212300410365) and Research and innovation funding project for Postgraduates of Henan Normal University (No. YL202019).

 \end{document}